\documentclass{article}
\usepackage[utf8]{inputenc}
\usepackage{cite}

\usepackage{amsmath,amssymb,amsfonts,amsthm}
\usepackage{bbm}
\usepackage{algorithmic}
\usepackage{graphicx}
\usepackage{tablefootnote}
\usepackage{algorithm,algorithmic}
\usepackage{hyperref}
\hypersetup{colorlinks=true,linkcolor=black,anchorcolor=black,citecolor=black}

\usepackage{textcomp}
\def\BibTeX{{\rm B\kern-.05em{\sc i\kern-.025em b}\kern-.08em
    T\kern-.1667em\lower.7ex\hbox{E}\kern-.125emX}}
\begin{document}

\newcommand{\bx}{{\mathbf{x}}}
\newtheorem{assumption}{Assumption}
\newtheorem{definition}{Definition}
\newtheorem*{notation*}{Notation}
\newtheorem{theorem}{Theorem}
\newtheorem{corollary}{Corollary}
\newtheorem{lemma}{Lemma}
\newtheorem{remark}{Remark}
\newtheorem{proposition}{Proposition}
\renewcommand{\algorithmicrequire}{\textbf{Input:}}
\renewcommand{\algorithmicensure}{\textbf{Output:}}

\newcommand{\beq}{\begin{equation}}
\newcommand{\eeq}{\end{equation}}
\newcommand{\beqa}{\begin{eqnarray}}
\newcommand{\eeqa}{\end{eqnarray}}
\newcommand{\beqas}{\begin{eqnarray*}}
\newcommand{\eeqas}{\end{eqnarray*}}
\newcommand{\bi}{\begin{itemize}}
\newcommand{\ei}{\end{itemize}}
\newcommand{\gap}{\hspace*{2em}}
\newcommand{\sgap}{\hspace*{1em}}
\newcommand{\vgap}{\vspace{.1in}}
\newcommand{\nn}{\nonumber}
\newcommand{\bRe}{\bar \Re}
\newcommand{\IR}{\makebox{\sf I \hspace{-9.5 pt} R \hspace{-7.0pt} }}
\newcommand{\defeq}{\stackrel{\triangle}{=}}

\setcounter{page}{1}
\def\eqnok#1{(\ref{#1})}

\newcommand{\tos}{\rightrightarrows}
\newcommand{\R}{\mathbb{R}}
\newcommand{\calS}{{\cal S}}
\newcommand{\cG}{{\cal G}}
\newcommand{\cL}{{\cal L}}
\newcommand{\cX}{{\cal X}}
\newcommand{\cC}{{\cal C}}
\newcommand{\cZ}{{\cal Z}}
\newcommand{\Z}{{\cal Z}}
\newcommand{\lam}{{\lambda}}
\newcommand{\blam}{{\bar \lambda}}
\newcommand{\A}{{\cal A}}
\newcommand{\tA}{{\tilde A}}
\newcommand{\tC}{{\tilde C}}
\newcommand{\cN}{{\cal N}}
\newcommand{\cK}{{\cal K}}
\newcommand{\norm}[1]{\left\Vert#1\right\Vert}
\newcommand{\ind}[1]{I_{#1}}
\newcommand{\inds}[1]{I^*_{#1}}
\newcommand{\supp}[1]{\sigma_{#1}}
\newcommand{\supps}[1]{\sigma^*_{#1}}
\newcommand{\simplex}[1]{\Delta_{#1}}
\newcommand{\abs}[1]{\left\vert#1\right\vert}
\newcommand{\Si}[1]{{\cal S}^{#1}}
\newcommand{\Sip}[1]{{\cal S}^{#1}_{+}}
\newcommand{\Sipp}[1]{{\cal S}^{#1}_{++}}
\newcommand{\inner}[2]{\langle #1,#2\rangle}
\newcommand{\Inner}[2]{\big \langle #1\,,#2 \big \rangle}

\newcommand{\ttheta}{{\tilde \theta}}
\newcommand{\cl}{\mathrm{cl}\,}
\newcommand{\co}{\mathrm{co}\,}
\newcommand{\argmin}{\mathrm{argmin}\,}
\newcommand{\cco}{{\overline{\co}}\,}
\newcommand{\ri}{\mathrm{ri}\,}
\newcommand{\inte}{\mathrm{int}\,}
\newcommand{\interior}{\mathrm{int}\,}
\newcommand{\bd}{\mathrm{bd}\,}
\newcommand{\rbd}{\mathrm{rbd}\,}
\newcommand{\cone}{\mathrm{cone}\,}
\newcommand{\aff}{\mathrm{aff}\,}
\newcommand{\lin}{\mathrm{lin}\,}
\newcommand{\lineal}{\mathrm{lineal}\,}
\newcommand{\epi}{\mathrm{epi}\,}
\newcommand{\sepi}{\mbox{\rm epi}_s\,}
\newcommand{\dom}{\mathrm{dom}\,}
\newcommand{\Dom}{\mathrm{Dom}\,}
\newcommand{\Argmin}{\mathrm{Argmin}\,}
\newcommand{\lsc}{\mathrm{lsc}\,}
\newcommand{\di}{\mbox{\rm dim}\,}
\newcommand{\EConv}[1]{\mbox{\rm E-Conv}(\Re^{#1})}
\newcommand{\EConc}[1]{\mbox{\rm E-Conc}(\Re^{#1})}
\newcommand{\Conv}[1]{\mbox{\rm Conv}(\R^{#1})}
\newcommand{\Conc}[1]{\mbox{\rm Conc}(\Re^{#1})}
\newcommand{\cball}[2]{\mbox{$\bar {\it B}$}(#1;#2)}
\newcommand{\oball}[2]{\mbox{\it B}(#1;#2)}
\newcommand{\Subl}[2]{{#2}^{-1}[-\infty,#1]}
\newcommand{\sSubl}[2]{{#2}^{-1}[-\infty,#1)}
\newcommand{\bConv}[1]{\overline{\mbox{\rm Conv}}\,(\R^{#1})}
\newcommand{\bEConv}[1]{\mbox{\rm E-C}\overline{\mbox{\rm onv}}\,(\Re^{#1})}
\newcommand{\asympt}[1]{{#1}'_{\infty}}
\newcommand{\sym}[1]{{\cal S}^{#1}}
\newcommand{\spa}{\,:\,}
\newcommand{\barco}{\overline{\mbox{co}}\,}
\newcommand{\tx}{\tilde x}
\newcommand{\ty}{\tilde y}
\newcommand{\tz}{\tilde z}
\newcommand{\tm}{\tilde m}
\newcommand{\mConv}[1]{\overline{\mbox{\rm Conv}}_\mu\,(\R^{#1})}
\newcommand{\la}{\langle\,}
\newcommand{\ra}{\,\rangle}
\def\iff{\Leftrightarrow}
\def\solution{\noindent{\bf Solution}. \ignorespaces}
\def\endsolution{{\ \hfill\hbox{%
      \vrule width1.0ex height1.0ex
          }\parfillskip 0pt}\par}
\def\NN{{\mathbb{N}}}

\title{Enhancing Convergence of  Decentralized  Gradient Tracking \\ under the KL Property}
\author{Xiaokai Chen
\thanks{School of Industrial Engineering, Purdue University, West Lafayette, IN 47906 (email: {\tt chen4373@purdue.edu}).}\qquad Tianyu Cao\thanks{School of Industrial Engineering, Purdue University, West Lafayette, IN 47906 (email:  {\tt cao357@purdue.edu}). }\qquad Gesualdo Scutari\thanks{School of Industrial Engineering, Purdue University, West Lafayette, IN 47906 (email {\tt gscutari@purdue.edu}). \\
This work has been supported by the ORN Grant N. N000142412751. }
	}
\date{December 12, 2024}

\maketitle

\begin{abstract}
\label{sec:abstract}
We study decentralized multiagent optimization over networks, modeled as undirected graphs. The optimization problem consists of minimizing a nonconvex smooth function plus a convex extended-value function, which enforces constraints or extra structure on the solution (e.g., sparsity, low-rank). We further assume that the objective function satisfies the Kurdyka-Łojasiewicz (KL) property, with given exponent   $\theta\in [0,1)$. The KL property  is satisfied by several (nonconvex) functions of practical interest, e.g., arising from machine learning applications; in the centralized setting, it permits to achieve strong convergence guarantees. Here we establish  convergence of the same type for  the notorious  decentralized gradient-tracking-based algorithm SONATA, first proposed in\cite{SunDanScu19}. Specifically, \textbf{(i)} when $\theta\in (0,1/2]$,  the sequence generated by SONATA  converges to a stationary solution of the problem at  R-linear rate; \textbf{(ii)} when  $\theta\in (1/2,1)$,  sublinear rate  is certified; and finally \textbf{(iii)}   when $\theta=0$, the iterates will either converge in a finite number of steps or converges at R-linear rate. This matches the convergence behavior of centralized proximal-gradient algorithms except when $\theta=0$. Numerical results validate our theoretical findings. 

{\bf Key words.}
distributed optimization, decentralized methods, nonconvex optimization, gradient-tracking, Kurdyka-Łojasiewicz, linear rate. 
\end{abstract}

\section{Introduction}
\label{sec:introduction}
This paper studies nonconvex,   nonsmooth optimization problem over networks of the following form:
\begin{equation}
\begin{array}{rl}
\underset{{x}\in\mathbb{R}^{d}}{\min} & u({x}) := f({x})+ r({x}),
\end{array}
\tag{P}
\label{problem}
\end{equation}
where\vspace{-0.2cm}
\begin{equation}
f({x}) := {\frac{1}{m}}\sum^{{m}}_{i=1}{f_{i}({x})}, 
\label{e1}
\end{equation}
\noindent
 with each $f_{i} : \mathbb{R}^{d} \to \mathbb{R}$ \ being the cost function of agent $i=1,\ldots, m$, assumed to be $L_{i}$-smooth (possibly nonconvex) and known only to agent $i$;\ $r : \mathbb{R}^{d} \to \mathbb{R}\cup \{-\infty,+\infty\}$ \ is a nonsmooth convex (extended-value) function, known to all agents, which can be used to enforce shared constraints or specific structures on the solution (e.g., sparsity). Agents are connected through a communication network, modeled as a fixed, undirected graph.

Problem~\eqref{problem} has a wide range of applications, such as network information processing, telecommunications and multi-agent control. Here, we are particularly interested in machine learning problems, specifically supervised learning. Examples include logistic regression, SVM, LASSO and deep learning. Each $f_{i}$ represents the empirical risk measuring the mismatch between the model (parameterized by ${x}$) to be learnt and the dataset owned by agent $i$; and  $r$ plays the role of the regularizer.  A vast range of these problems possess some structural properties, notably in the form of some growth property, which   can be leveraged to enhance convergence of decentralized solution methods.
In this paper we target  the so-called \textit{Kurdyka-Łojasiewicz} (KL) property, due to its vast range of applicability. The class of functions satisfying the KL property is ubiquitous, ranging from applications  in optimization \cite{liu2019quadratic,attouch2009convergence,attouch2010proximal,attouch2013convergence,bolte2014proximal,frankel2015splitting,li2018calculus}, machine learning (logistic regression, LASSO,  and the Principal Component Analysis (PCA) are notorious examples),   deep neural network~\cite{lau2018proximal},  dynamical
systems, and partial differential equations; see~\cite{BDLM10} and the references therein.  

The KL property was first introduced by Łojasiewicz~\cite{Lo63} for  real
analytic functions, then Kurdyka extended it to functions defined
on a generalization of semialgebraic and subanalytic geometry called the o-minimal structure~\cite{Ku98}. More recently, it was  extended to nonsmooth subanalytic
functions by Bolte et al.~\cite{bolte2007lojasiewicz}. In its simplest form~\cite{attouch2010proximal}, we say that a $C^{1}$ function $f:\mathbb{R}^d\to \mathbb{R}$       satisfies the KL property with exponent $\theta\in [0,1)$ at $\bar{x}\in \mathbb{R}^d$   (possibly critical) if there exists    $\theta\in[0,1)$ and some problem-dependent $c>0$ such that   
\begin{equation}\label{eqn420}
  ||\nabla f(x)||\, \geq c \, |f(x)-f(\bar{x})|^{\theta}.  
\end{equation} 
Ensuring a lower bound on the gradient magnitude  in the proximity of $\bar x$, this condition stipulates that the function  $f$ 
  cannot be arbitrarily flat around  $\bar x$. As for its nonsmooth counterpart, one can refer to Sec. \ref{subsec:KL} for more details.

 In the recent years, the KL property has been successfully leveraged to enhance
   convergence of several   solution methods for nonconvex, nonsmooth, {\it centralized} optimization problems;  examples include inexact first-order methods~\cite{attouch2013convergence},  alternating (proximal) minimization     algorithms~\cite{attouch2010proximal}, and parallel methods~\cite{cannelli2019asynchronous}.  In particular,   under the KL,     the {\it entire} sequence generated by these algorithms  is proved to converge to a stationary solution (in contrast with the much weaker guarantees in the absence of KL,  certifying subsequence convergence). The convergence rate  is dictated by the \textit{KL exponent} $\theta$, namely: when $\theta\in(0,1/2]$,  linear   rate is certified  whereas  sublinear convergence is proved when $\theta\in(1/2,1)$; finally, when $\theta=0$,   the sequence  converges in a {\it finite} number of iterations. 

 \begin{figure}[h]
\hspace*{-0.01in}
\centering
\includegraphics[scale=0.33]{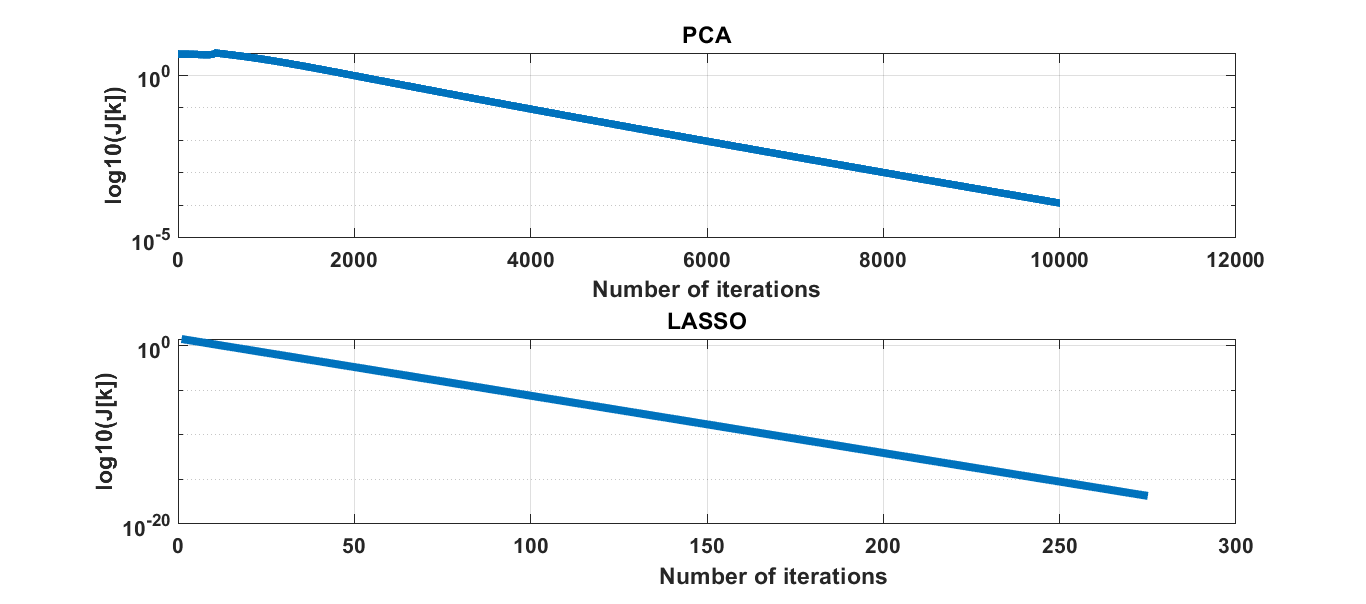}
\caption{Distance of the iterates ($x_i^{\nu}$, $i\in[m]$ ) from  a stationary solution ($x^*$) of the PCA (upper panel) and LASSO (lower panel) problems,  defined as $\mathcal{P}^\nu := \sqrt{\sum_{i=1}^m\|x_i^{\nu}-x^*\|^2}$, versus the iterations $\nu$. }\vspace{-0.5cm}
\label{fig1} 
\end{figure}

    This leads to the natural question whether these strong convergence guarantees    can be inherited by   decentralized algorithms. To the best of our knowledge, the existing literature on distributed algorithms does not provide a satisfactory answer (see Sec.~\ref{subsec:litrev} for a detailed review of the state of the art). Here we only point out that none of the existing analyses of decentralized algorithms \cite{scutari2019distributed,Ye20,LoScu15,SunDanScu19,shi2015extra}  applicable to problem (\ref{problem}) exploit the KL property  of $u$.   As a consequence, if nevertheless  invoked, they  
would  predict much more pessimistic convergence rates than what is certified in the centralized setting. For instance, when applied to \eqref{problem} under the KL   with   $\theta\in (0,1/2]$, they would claim
{\it sublinear} convergence of some  surrogate distance from stationarity, which contrasts the much more favorable    {\it linear} convergence of the {\it entire} sequence  established for a variety of centralized first-order methods applied to \eqref{problem}.  These pessimistic predictions are confuted by   experiments. Two   examples are reported in   Fig.~\ref{fig1}, where the the decentralized algorithm  SONATA \cite{scutari2019distributed} is applied to two nonconvex instances of \eqref{problem}  over a network modeled as an Erdos-Renyi  graph--the LASSO problem with SCAD regularizer ~\cite{li2018calculus} (cf. Sec.~\ref{subsec:KL}) and the PCA problem. The objective functions of both problems satisfy the KL property with exponent $1/2$. The figures plot  the distance of the SONATA's iterates  from a stationary solution versus number of iterations, certifying    
%$G(m,p)$, where $p=\textcolor?$ is the probability of including one edge (independently from every other edge). \textcolor{red}{say which graph youy simulated}. 
%The objective functions in both problems satisfy the KL property with exponent $\theta=1/2$, and 
{\it linear} convergence of the sequence for both problems. This fact has no theoretical backing.

 This paper fills  this gap. 
    As a case-study, we focus on the gradient tracking-based decentralized scheme SONATA, firstly proposed in~\cite{SunDanScu19}. We provide a full characterization of its  convergence rate,  for the entire   range of the KL exponent $\theta\in [0,1)$,  matching   convergence results  of the  proximal gradient algorithm in the centralized setting.  \vspace{-0.1cm}

\subsection{Main contributions}
\label{subsec:contribution}
\noindent Our technical contributions can be summarized as follow.  \smallskip

 \noindent $\bullet$ \textbf{Convergence rate under the KL of $u$:}  We establish convergence of the sequence generated by   SONATA  to stationary solutions of  \eqref{problem}, and characterize its convergence rate, under different values of the KL exponent of  of $u$. More specifically, 
 \begin{itemize}
     \item [i)] When  $\theta\in(0,1/2]$, the sequence generated by SONATA is proved to converge {\it R-linearly}. The number of communication rounds    to reach an $\epsilon$-stationary solution (under proper tuning)  reads $$\mathcal{\tilde{O}}\left(\frac{L}{\kappa^{1/\theta}}\frac{1}{1-\rho}\text{log}(1/\epsilon) \right),$$   where $L:= \frac{1}{m}\sum_{i=1}^{m}L_{i}$, with $L_i$ being   the smoothness constant of $f_i$;   $\kappa$ is a  parameter  related with the KL property of $u$ (see Def.~\ref{def-KL}), and  $\rho$   represents the network connectivity (see \eqref{eq:rho} for the formal definition).   The   $\mathcal{\tilde{O}}$ notation hides log-dependencies on $L$, $\kappa$, $\theta$ and $L_{\text{mx}}:=\max_{i\in[m]}L_i$.  \smallskip 
     \item [ii)] For   $\theta\in(1/2,1)$, sublinear convergence of the sequence is established, yielding  rates for $\epsilon$-stationarity of the order $$\mathcal{O}(\epsilon^{-\frac{2\theta-1}{1-\theta}}).$$
     \item[iii)] When  $\theta=0$, the  sequence either convergences to a stationary solution in a  finite number of iterations or converges at R-linear rate independent of $\kappa$.
 \end{itemize}    Notably,  the rates in (i) and (ii)  match those of the centralized proximal gradient algorithm (up to universal constants).  \smallskip 
 
 \noindent $\bullet$ \textbf{New convergence analysis:} We introduce a new line of analysis  that explicitly leverage the KL property, departing from traditional techniques     employed    to study   centralized and decentralized  algorithms, as detailed next.

Classical proofs--establishing linear convergence of first-order methods under the KL property (with $\theta=1/2)$  of the objective function--focused primarily on centralized settings~\cite{attouch2009convergence,attouch2010proximal,attouch2013convergence,bolte2014proximal,frankel2015splitting,li2018calculus,cannelli2019asynchronous}. They strongly rely on the objective function's monotonic decrease along algorithmic trajectories.  However, when it comes to decentralized algorithms,     this monotonicity    is disrupted by disagreements among agents' iterates, introducing perturbations that impair the descent of the  objective function $u$. This issue is addressed in \cite{JW18,Amir18} by constructing a Lyapunov function that suitably combines agents' objectives $f_i$'s with  consensus errors, and monotonically decreases along the algorithm trajectories, albeit limited to unconstrained, smooth instances of (P).     Assuming this Lyapunov function is a KL function  with $\theta=1/2$, linear convergence of the decentralized algorithms  \cite{JW18,Amir18}  was  proved, mirroring techniques used in centralized settings~\cite{attouch2009convergence,attouch2010proximal,bolte2014proximal,frankel2015splitting,li2018calculus,attouch2013convergence}. Unfortunately, the KL property of the Lyapunov function or its exponent value {\it do not transfer to the objective function} $u$  of the optimization problem,    and vice versa. This is because the KL property is not closed with respect to the operations used to  build  the Lyapunov function from the objective function. Consequently, 
 the challenge of establishing convergence guarantees for decentralized algorithms under the KL property of the objective function--comparable to those certified in centralized settings--remains unresolved.

Our novel approach hinges   on  the KL property of the objective function $u$. First, we establish asymptotic convergence of the objective function gap and consensus and tracking errors, along  the iterates produced by the SONATA algorithm, through the monotonically decaying trajectory of a suitably constructed Lyapunov function. This guarantees that, after a sufficiently large but finite number of iterations--such that  the function value gap falls below a critical threshold--the KL property of the objective function can be engaged. Consequently   convergence of the entire sequence to a critical point of $u$ is established,  at enhanced   convergence  rate.   \vspace{-0.2cm}
        
\subsection{Related works}\label{subsec:litrev} 
%The analysis framework based on the Kurdyka-Łojasiewicz property has been successfully and widely utilized in the centralized gradient-based algorithms over nonconvex nonsmooth problems thanks to several compelling reasons. Firstly, it is practical: the KL inequality is a local geometric property that is proved to be satisfied for a vast and ubiquitous class of functions~\cite{liu2019quadratic,attouch2010proximal,attouch2013convergence,GuoTin17}, which covers a large set of applications including a wide range of machine learning problems and even some variants of deep neural network training ~\cite{lau2018proximal}. Moreover, 

\noindent \textbf{Centralized setting:}   
The  KL  property has been widely utilized in the convergence analysis of centralized optimization methods. The pioneering study of~\cite{PRB05} marks the initial use of the KL property to demonstrate sequence convergence to  stationary solutions of a variety of  algorithms satisfying certain decent properties. However,  no convergence rate analysis was established.  The followup works ~\cite{attouch2009convergence,attouch2010proximal,bolte2014proximal,frankel2015splitting,li2018calculus,attouch2013convergence,cannelli2019asynchronous} did provide  a convergence rate analysis of these  algorithms, %minimizing KL functions satisfying the sufficient decent and relative error conditions, 
specifically: the proximal gradient~\cite{attouch2009convergence},  the alternating proximal minimization algorithm~\cite{attouch2010proximal}, the inexact Gauss–Seidel method~\cite{attouch2013convergence}, the Proximal alternating linearized minimization algorithm~\cite{bolte2014proximal},  the alternating forward-backward splitting method~\cite{frankel2015splitting}, and the (parallel) asynchronous method known as FLEXA    \cite{cannelli2019asynchronous}.   %Further developments are made in~\cite{AttBol13}, which establishes linear convergence to the critical point of structured nonsmooth functions by examining the alternating proximal minimization algorithm. 
%In the context of parallel computing, the parallel scheme FLEXA is shown to achieve linear convergence to the stationary solution for the entire iterate sequence~\cite{cannelli2019asynchronous}. 
 %a modified analytical framework is proposed in~\cite{bento2024convergence} to handle applications satisfying different versions of KL properties, in which the subdifferential of the objective function is replaced respectively by symmetric subdifferential (symmetric KL property) and convexified subdifferential (strong KL property), making the framework more general.  
 %From another perspective, there have been works~\cite{qiu2024kl,ochs2014ipiano,qian2023convergence} extending the KL-based analysis into a wider classes of algorithms that may not strictly satisfy the sufficient decent and relative error conditions that appeared in the traditional analysis. %which takes a variety of stochastic and distributed methods into account.\textcolor{red}{you have not discussed distributed methods yet} 
 %Notably, in~\cite{qiu2024kl}, they comprehensively relaxed the conditions by tolerating an additional appropriately-defined error term as well as accepting a nonmonotone flow, namely, the decent can happen only in a subset of the iterates.  \smallskip 
Other studies \cite{bento2024convergence, ochs2014ipiano, qian2023convergence, qiu2024kl} have expanded the class of KL functions and relaxed algorithmic constraints while maintaining the same strong convergence guarantees. Specifically, \cite{bento2024convergence} extends the KL property to include symmetric and strong KL variations; and   \cite{ochs2014ipiano, qian2023convergence, qiu2024kl} modify algorithm requirements by permitting an additive error term in the relative error condition \cite{ochs2014ipiano} and accommodating a nonmonotone descent flow \cite{qian2023convergence}. Notably, \cite{qiu2024kl} integrates both modifications in their analytical framework.\smallskip

 \noindent\textbf{Decentralized setting:} Decentralized algorithms for various instances of Problem~\eqref{problem} have received significat attention in the last few years %The use of the KL property to enhance convergence of decentralized algorithms has not been satisfactorily explored. In fact, several algorithms have been proposed to solve nonconvex optimization problems over networks, including those the form \eqref{problem}~
\cite{hong2022divergence,bianchi2012convergence,di2016next,wai2017decentralized,tang2018d,hong2017prox,scutari2019distributed,tatarenko2017non,zhu2012approximate}. Specifically, early studies 
~\cite{hong2022divergence,tang2018d,hong2017prox,tatarenko2017non} consider smooth objectives (i.e., $r=0$)   whereas~\cite{di2016next,wai2017decentralized,scutari2019distributed,bianchi2012convergence,zhu2012approximate}  extended to constraints or composite structures, under the assumption of bounded (sub)gradient of the objective loss along the iterates of the algorithm.       This restriction has been removed in~\cite{zhu2012approximate,scutari2019distributed}, with  \cite{scutari2019distributed} providing also  a convergence rate analysis (applicable to time-varying networks).  %Still, no assumption on any possible regularity properties is made and thus it can only provide a sublinear rate of convergence and such rate is not proved for the iterate sequence but only for a suitably defined merit function.  
None of the aforementioned works   exploit any  growing property of the objective function, such as the KL, if any. This leads to pessimistic convergence guarantees (asymptotic convergence only or sublinear convergence rates), which contrasts with the results in the centralized setting discussed above and is inconsistent with the numerical results presented in Fig.~1 (Sec.~\ref{sec:introduction}) and Sec.\ref{sec:numeric result}.  

Works exploiting explicitly some (postulated) function growth  to enhance  convergence guarantees of decentralized algorithms applied to special  instances of (P), include \cite{SLG19,XSTKT19,XSTTK20,Ye20,JW18,Amir18}. Specifically, linear rate of the considered decentralized algorithms is proved  under the restricted secant condition~\cite{SLG19,XSTKT19}, the Polyak-Łojasiewicz (PL) condition~\cite{XSTTK20}, and the Luo-Tseng error bound condition~\cite{Ye20}. However,   in the nonconvex setting, all these conditions are more stringent than the  KL property (with exponent 1/2) \cite{PL18}.   Furthermore, convergence     techniques  therein   closely mirror those used for strongly convex functions, which are not useful in the setting considered in this paper.  
 
On the other hand, while studies such as \cite{Amir18,JW18} have utilized the KL property in the convergence analysis of some decentralized algorithms,  they have not conclusively achieved the desired outcomes. As discussed in Sec \ref{subsec:contribution}
, these works postulate the KL property of specifically constructed Lyapunov functions that meet the necessary conditions for convergence, rather than of the original objective functions. The link between the KL property of the objective function and such Lyapunov functions remains unclear, highlighting a gap in the current literature. \vspace{-.2cm}

%In contrast to all the existing results, we establish linear convergence rate of the entire iterate sequence to a stationary solution under the assumption that KL property is satisfied by only the objective function. To the best of our knowledge, there is no such scheme in the distributed optimization literature which can handle the exact problem formulation (\ref{problem}) under as general assumptions as considered in this work. 

\subsection{Notation and paper organization}
Throughout the paper, we will use the following notation.  
%We denote the $d$-dimensional real space as $\mathbb{R}^d$, $\mathbb{R}_+^d:=\{x\in\mathbb{R}^d:x\succeq 0\}$. Similarly, $\mathbb{N}$ represents the set of all the nonnegative integers and $\mathbb{N}_{+}:=\mathbb{N}\backslash\{0\}$. 
For any integer $m$, we write  $[m]:=\{1,2,\cdots,m\}$. We user  the convention that $0^0=0$. We denote by $[c_1\leq(<) u\leq(<) c_2]:=\{x:c_1\leq(<) u(x)\leq(<) c_2\}$   the level set of the function $u$ at $c\in\mathbb{R}$. We will use capital letters to represent matrices. In particular, $1$  denotes the vector of all ones (whose dimensions are clear from the context);  $J:=11^{\top}/n$ is the projection onto the consensus space.   We use $\|\cdot\|_p$ to denote the $\ell_p$-norm of any input vector  ($\|\cdot\|$ will be the Euclidean norm) whereas   $\|\cdot\|_p$ represents  the operator norm induced by the $\ell_p$-norm when the input is a matrix (with  $\|\cdot\|$ being  the Frobenius norm). 
%$\|\cdot\|_p$ denotes the $\ell^p$-norm for vectors and operator norm induced by $\ell^p$-norm for matrices. If not specified, $\|\cdot\|$ stands for the Euclidean norm when the argument is a vector and  Frobenius norm when the argument is a matrix.
 Several operators appear in the paper.  %For a function $r:\mathbb{R}\rightarrow\mathbb{R}\cup\{-\infty,\infty\}$, define $\textit{dom}(r):=\{x\in\mathbb{R}^d, r(x)<\infty\}$.
 Given a proper, nonconvex function $f$,   $\partial f(x)$     denotes the  (limiting) subdifferential    of $f$ at $x$~\cite[Def. 8.3(b)]{rockafellar1998variational}.  Given $x\in\mathbb{R}^d$ and $r:\mathbb{R}\rightarrow\mathbb{R}\cup\{-\infty,\infty\}$, $\texttt{prox}_{\alpha r}(x)$ denotes the     proximal operator, defined as  $$\texttt{prox}_{\alpha r}(x):=\displaystyle{\text{argmin}}_{y\in \mathbb{R}^d}r(y)+\frac{1}{2\alpha}\|x-y\|^2.$$
 %In addition, we denote the distance from an $x \in \mathbb{R}^n$ to $D$ by $\text{dist}(x, D) = \inf_{y\in D} \|x-y\|$, and the set of points in $D$ that achieve this infimum (the projection of x onto D) is denoted by $\text{Proj}_D(x)$.

 %We use $||\cdot||_p$ to denote the $\ell^p$-norm. If not specified, $\|\cdot\|$ stands for the Frobenius norm when the argument is a matrix and the Euclidean norm when the argument is a vector. We use $1$ \textcolor{red}{use just $1$} to denote the vector/matrix with all the elements being $1$ and write $ J:=\frac{1}{m}11^{\top}$ to denote the averaging matrix. For $x\in\mathbb{R}^d$, the standard proximal operator is denoted as $\textit{prox}_{\alpha r}(x):=\textit{argmin}_{y} r(y)+\frac{1}{2\alpha}\|x-y\|^2.$  For simplification, we use $[m]$ to denote the set containing integers from $1$ to $m$. \textcolor{red}{I think there are many missing items, can you please go again over the entire paper and proofs and put here any symbol in the sense I mentioned that you encounter and use}

The rest of the paper is organized as following.  Sec.~\ref{sec:problem} introduces the problem formulation  along with the underlying  assumptions. Sec.~\ref{sec:asymptotic convergence} presents the asymptotic convergence analysis of SONATA, which is instrumental to engage the   KL property to enhance the convergence rate.  Sec.~\ref{sec:rate} contains the main technical result of the paper: the convergence rate analysis of SONATA under the KL property.    Finally,  some numerical experiments are presented   in Sec.~\ref{sec:numeric result}.
 
\section{Problem Setup and Background}
\label{sec:problem}
\noindent We study  the Problem \eqref{problem} over a communication network. Followings are standard assumptions on \eqref{problem}. 
\begin{assumption}[objective function]
\hfill
\label{ass:function}
\begin{itemize}
    \item [(i)] Each $f_i:\mathbb{R}^d\rightarrow \mathbb{R}$ is continuously differentiable, and $\nabla f_i$ is $L_i-$Lipschitz, with $L_i<\infty$;
    \item[(ii)] $r:\mathbb{R}^d\rightarrow\mathbb{R}\cup\{-\infty,\infty\}$ is convex, proper and lower-semicontinuous; 
    \item[(iii)] $u:\mathbb{R}^d\rightarrow \mathbb{R}$ is lower bounded by some $\underline{u}\in\mathbb{R}$. 
\end{itemize}
\end{assumption}
Associated with Assumption~\ref{ass:function}, we define the following quantities used throughout the paper. 
\begin{equation}
    L_\text{mx}%{\textcolor{red}{\text{mx}}}
    :=\max_{i\in[m]} L_i, \quad \text{and}\quad L:=\frac{1}{m}\sum_{i=1}^m L_i. 
\end{equation}

  The communication network of the agents is modeled as a time-invariant undirected graph $\mathcal{G}:=(\mathcal{V},\mathcal{E})$, with the vertex set $\mathcal{V} := \{ 1,...,m \}$  and the edge set $\mathcal{E}:=\{(i,j) | i,j\in\cal{V}\}$ representing the set of agents and the communication links, respectively. Specifically, $(i,j)\in\mathcal{E}$ if and only if there exists a communication link between agent $i$ and $j$. We make the blanket assumption that  $\mathcal{G}$ is connected.
 
\subsection{The KL property: definitions and illustrative examples}
\label{subsec:KL}
\noindent In this section, we formally introduce the KL property for $u$ along with some examples of KL functions arising from several  applications.  

%The  definition of KL property is   as follows. We use the following notation, given $u : \mathbb{R}^{d}\to \mathbb{R} \cup \{-\infty,+\infty\}$, let  $[a<u< b] := \{x\in\mathbb{R}^{d} : a<u(x)< b\}$. 
%\begin{definition}{(KL Property~\cite{HJB13})}
%\noindent
%The function $u : \mathbb{R}^{d}\to \mathbb{R} \cup \{-\infty,+\infty\}$  has KL property at $x^{*}\in$ {dom}  $\partial u$ if there exists $\eta \in (0,+\infty]$, a neighborhood $\mathcal{V}_{x^{*}}$, and a continuous concave function $\phi : [0,\eta)\to \mathbb{R}_{+}$ such that :
%\begin{enumerate}
    %\item[(i)] $\phi(0) = 0$; $\phi$ is $\mathcal{C}^{1}$ on $(0,\eta)$,
    %\item[(ii)] $\forall s\in (0,\eta), \phi'(s)>0$,
    %\item[(iii)] $\forall x\in \mathcal{V}_{x^{*}}\cap [u(x^{*})<u< u(x^{*})+\eta]$, the KL inequality holds :
    %\begin{equation}\label{eq350}
     %   \phi'\left(u(x)-u(x^{*})\right) \text{dist}(0,\partial u(x)) \ge 1
    %\end{equation}
%\end{enumerate}
%A proper lower-semicontinuous function $u$ is called KL if it satisfies the KL inequality at every point in dom $\partial u$.
%\end{definition}
The general definition of the  KL  can be found in~\cite{attouch2013convergence}. Here we focus more specifically on   functions that are   sharp up to some reparametrization, using  the  so-called \textit{desingularizing} function, denoted by $\phi:\mathbb{R}\rightarrow\mathbb{R}$. This function  turns a \textit{singular} region--a region where the gradients are arbitrarily small--into a \textit{regular} region--where the gradients are bounded away from zero. A widely used  desingularizing function is  $\phi(s) := cs^{1-\theta}$, for some  $\theta\in [0,1)$ and $c>0$, firstly introduced in~\cite{Lo63}. More specifically, the KL property equipped with this desingularizing function reads as follows.  

\begin{definition}{(KL property)\cite{li2018calculus}}\label{def-KL}
A proper closed function $f$ satisfies the KL property at  $\bar x\in \text{dom} f$   with exponent $\theta\in[0,1)$ if there exists a neighborhood $E$ of $\bar x$ and  parameters $\kappa,\eta\in (0,\infty)$ such that 
\begin{equation}
\label{eq353}  \|g_x\|\geq \kappa(f(x)-f(\bar x))^{\theta},
\end{equation} for all $x\in E\cap [f(\bar x)<f<f(\bar x)+\eta]$   and   $g_x\in\partial f(x)$. %, where $\partial f(x)$ is  the (limiting) subdifferential ((See Def.~\ref{def:subdifferential} in Appendix~\ref{subsec:A})) of $f$ at $x$ . 
We call the function $f$ a KL function with exponent $\theta$ if it satisfies the KL property at any point $\bar x\in \text{dom} \partial f$, with the same exponent $\theta$.
\end{definition} 
%Notice that in the above definition, the value of parameters $\kappa$, $\eta\in(0,\infty)$ depend on the reference point $\bar x\in \text{dom} f$. Since our  interest falls on the critical points of the functions $u$ in (\ref{problem}), to better characterize the parameter $\kappa$ in the later analysis, we denote $(\kappa,\eta)$ the \textit{KL parameters at $\bar x$} and specifically denote $\kappa$ as the \textit{KL coefficient} if $\bar x$ is a stationary point of the function $f$.

%\textcolor{red}{1. List basic functions satisfying the property (same spirit as   Lemma 3); 2.   key operations preserving the KL and rules to compute the exponent (lemma 4, check if any of the others, e.g. in [8] apply, Lemma 6 at the end); 3. classes of functions as by products of those rules, Prop. 1 and Prop 2; 4. concrete applications (those that we have now). }

\subsubsection{Some illustrative examples} \label{sub-examples} We listed some motivating examples of functions $u$ in the form~\eqref{problem} that satisfies the KL property, with specified exponent $\theta\in[0,1)$.

\smallskip
 \textit{{\bf (i)} Sparse Linear Regression (with $\ell_1$ regularization):}  The sparse linear regression problem consists in estimating  a $s$-sparse parameter $x^{*}\in \mathbb{R}^d$ via a set of linear measurements, corrupted by noise, that is, 
   $y=Ax^{*}+w$, where $y\in \mathbb{R}^N$ is the vector of measurements,  $A\in\mathbb{R}^{N\times d}$ is the design matrix, and  $w\in \mathbb{R}^N$ is the observation noise. Assuming each agent in the network owns a subset $y_i\in \mathbb{R}^n$ of the overall $N$ measurements $y$ along with the design matrix $A_i\in \mathbb{R}^{n\times d}$, with each $y_i$ (resp. $A_i$) such that $y=[y_1^\top,\ldots, y_m^\top]^\top$ (resp. $A=[A_1^\top,\ldots, A_m^\top]^\top$), the decentralized estimation of $x^\star$ via the LASSO estimator is an instance of Problem (P), with $f_i(x):=({1}/{2})\|A_ix-b_i\|^2$ and 
      $ r(x) := \lambda||{x}||_{1}$,  $\lambda>0$. The overall resulting loss $u$ is KL with  exponent $\theta={1}/{2}$~\cite{li2018calculus}. \smallskip

 \textit{{\bf (ii)} Sparse linear Regression with SCAD regularization: } %SCAD regularized least-square is an adaptive variant of the LASSO that is developed to reduce bias for large regression coefficients. In this model, we have %$f({x}) := ||{A}{x}-{b}||^{2},\  {x}\in\mathbb{R}^{d}$, ${A}\in\mathbb{R}^{mn\times d}$ is non-full rank matrix; $n$ is the number of samples available to agent $i$;  $r(x):=\sum^{d}_{i=1}{p(x_{i})}$, where $p(\cdot)$ is given by 
%\[ p(x) :=  \begin{cases}  
          % \frac{2\theta'}{a'+1}|x| & |x|\le \frac{1}{\theta'} \\
           %\frac{-x^{2}\theta'^{2}+2a'\theta'|x|-1}{a'^{2}-1} & \frac{1}{\theta'}<|x|\le\frac{a'}{\theta'} \\
           %1  & |x|>\frac{a'}{\theta'}
         %\end{cases}
%\]
%\noindent
%with $a'>0, \theta'>0$ are positive real numbers. Although SCAD regularizer is \textbf{nonconvex}, it satisfies the KL inequality with exponent ${1}/{2}$.
 The LASSO formulation, as discussed in the previous example,  tends to yield   biased estimators for large regression coefficients. To address this issue, the literature suggests replacing the $\ell_1$ norm with nonconvex nonsmooth regularizers, such as the smoothly clipped absolute deviation (SCAD) penalty. The SCAD penalty can be rewritten as a  Difference-of-Convex  function~\cite{ahn2017difference}, that is, ${r}_+(x)- {r}_{-}(x)$, where %${r}_+(x) := \sum_{k=1}^{d}p(x_k)$
 ${r}_+(x)=\lambda \|x\|_1$ and $ {r}_{-}(x) := \sum_{k=1}^d p(x_k)$, with  $p: \mathbb{R}\rightarrow\mathbb{R}$  defined as
\[ %q(x) := \frac{2\theta}{a+1}|x|,\,\,
p(x) := \begin{cases}  
           0, & \text{if } |x|\le \lambda \\
           \frac{ (|x|-\lambda)^2}{2(a^{2}-1)}, & \text{if }  \lambda<|x|\leq {a}\lambda \\
          \lambda |x|-\frac{(a+1)\lambda^2}{2},    & \text{if }  |x|\geq  {a}\lambda,
         \end{cases}
\] where $a>1$ and $\lambda>0$ are hyperparameter to properly tune. 
 It is not difficult to check that the decentralized sparse linear regression problem using the SCAD penalty is an instance of \eqref{problem}, with $f_i(x) :=  (1/2)\|A_i x- b_i\|^2 -{r}_{-}(x)$ and $r(x) :=  {r}_+(x)$.   
The objective function  satisfies the KL  property with exponent ${1}/{2}$~\cite{li2018calculus}.\smallskip

\textit{{\bf (iii)} Logistic  Regression:}  
Logistic regression aims to  estimate a parameter $x^*\in\mathbb{R}^d$ within a logistic model, where the log-odds of an event $a\in \mathbb{R}$ are modeled as the inner product between the model parameter  $x^*$ and the input features $b$. Specifically, the log-odds of   $a$ are given by   $$\log\frac{P_{x^\star}(a)}{1-P_{x^\star}(a)} = -b^{\top} {x^\star},$$ where $P_{x^\star}(a)$ denotes the predicted probability of  $a\in\mathbb{R}$, depending  on $x^*$.  Given the data samples $(a_k, b_k)_{k=1}^N$ equally distributed across a network of agents, with each agent $i$ owning the subset $(a_{ij}, b_{ij})_{j=1}^{n}$, the decentralized estimation of   $x^*$ is obtained by maximizing the log-likelihood of $P_x$ (with respect to $x$). This formulation is an instance of Problem $\eqref{problem}$, with each  $f_i(x) :=\frac{1}{n} \sum_{j=1}^{n}\log\left(1+{\rm exp}(b_{ij}^{\top}x)\right)$ and $r(x) \equiv 0$. The objective function  satisfies the KL  property with exponent ${1}/{2}$~\cite{li2018calculus}.   \smallskip

%it has the following DC (Difference of Convex functions)  structure: 

%\noindent
%$p(x) = \eta(\theta')|x| - (\eta(\theta')|x| - r(x))$, with $\eta(\theta') = \frac{2\theta'}{a'+1}$. Let $g^{+}(x):=\eta(\theta')|x|$ and $g^{-}(x) := \eta(\theta')|x| - r(x)$. Note that $g^{-}(x)$ is convex and has Lipschitz continuous first derivative.

%\noindent
%This allows us to cast this problem in the framework of problem \eqref{problem} with $f({x}) := ||{A}{x}-{b}||^{2} - \sum^{n}_{i=1}{g^{-}(x_{i})}$ and $r(x) := \sum^{n}_{i=1}{g^{+}(x_{i})}$.

\textit{ {\bf (iv)} Principal Component Analysis (PCA): } Given a (standardized) data set  $\{a_k\}_{k=1}^{N}$ ($a_k\in\mathbb{R}^d$), the (sample) instance of the PCA is to extract the leading eigenvector $x^*\in\mathbb{R}^d$ of the sample covariance matrix  $({1}/{N})\sum_{k=1}^{N}a_k a_k^{\top}$. Consider a network of agents, each one owning the subset $\{a_{ij}\}_{j=1}^{n}$. The decentralized estimation of $x^\star$ over the network can be formulated as  \eqref{problem}, with each $f_i(x) = -\frac{1}{n}\sum_{j=1}^{n}\|a_{ij}^{\top}x\|^2$, and $r(x) = \iota_{\mathcal{K}}(x)$. Here, $\mathcal{K} := \{x\in\mathbb{R}^d:\|x\|_2\leq 1\}$, and $\iota_{\mathcal{K}}(\bullet)$ denotes the indicator function of the convex set $\mathcal{K}$ (hence convex).  The objective function  satisfies the KL  property with exponent ${1}/{2}$~\cite{li2018calculus,li2017convergence}. \smallskip

\textit{{\bf (v)} Phase Retrieval: }
 Phase retrieval focuses on recovering   a signal $x^*\in\mathbb{R}^d$ via the measurements  $y_k = |a_k^{\top}x^*|+w_k$, $k\in [N]$, where $y_k\in\mathbb{R}$ is the observed magnitude, $a_k\in\mathbb{R}^{d}$ is the measurement vector and $w_k\in\mathbb{R}$ is the noise. Assuming each agent $i$ owns a subset $n$ of all $N$ measurements, $(y_{ij})_{j=1}^{n}$, and   vectors $(a_{ij})_{j=1}^{n}$,    the decentralized recovering of $x^*$ can be formulated as Problem \eqref{problem}, with $f_i(x)  := \frac{1}{2n}\sum_{j=1}^{n}\left( |a_{ij}^{\top} x|-y_{ij}\right)^2$ and $r \equiv 0$. The objective function $u$ satisfies the KL property with $\theta =  {1}/{2}$ \cite{zhou2017characterization}.

\smallskip

 %\textit{{\bf (vi)} Blind Deconvolution: }
%Blind deconvolution consists to recover two signals $s_1, s_2\in\mathbb{C}^N$, from their circulant convolution \cite{ahmed2013blind}. The output in the frequency domain can be written as $y = {\rm diag}(\hat{s}_1)\hat{s}_2$, where $\hat{s}_1$ (resp. $\hat{s}_2$) is the discrete Fourier transform of $s_1$ (resp. $s_2$). Additionally, assume $\hat{s}_1,\hat{s}_2$ lie in some subspaces, specifically  $\hat{s}_1 = {\rm conj}(Az^*)$ and $\hat{s}_2 = Bh^*$, where ${\rm conj}(.)$ denotes the element-wise conjugate operator,  $A = [a_1,\cdots, a_N]^{\mathsf{H}}, B = [b_1,\cdots,b_N]^{\mathsf{H}}$ are the measurement matrices, and $z^* \in\mathbb{C}^{L}$ and $h^*\in\mathbb{C}^{K}$ are some unknown coefficients. Then original problem is transformed to recovering $x^*:= \left((z^*)^{\top}, (h^*)^{\top}\right)^{\top}$ from $N$ binlinear measurements, in the form $y_k = b_k^\mathsf{H} h^* (z^*)^\mathsf{H}a_k$ ($k\in[N]$).  In the decentralized setting,   each agent in the network owns a subset of $n$ measurements; the resulting optimization problem can be written in the form  \eqref{problem}, with $f_i(x) := \frac{1}{2n}\sum_{j=1}^{n}|b_{ij}^{\mathsf{H}}h^*(z^*)^{\mathsf{H}}a_{ij}-y_{ij}|^2$ and $r\equiv 0$. The objective function $u$ is a    KL function,  with $\theta =  {1}/{2}$ \cite{li2019rapid}.\smallskip

%\noindent In the following example, we will see a practical application satisfying the KL property with exponent $\theta\in[0,1)$ (see~\cite{AttBol13}) 

\textit{{\bf (vi)} Deep Neural Network (DNN): }  
Consider a Deep Neural Network (DNN)   composed of multiple layers, each of which is equipped with weights $H^{l}\in\mathbb{R}^{d_{l}\times d_{l-1}}$ where $d_{\ell}, d_{\ell-1}$ are respectively output and input dimensions for the $\ell$th layer, $\ell\in\{1,\cdots,\Gamma\}$ where $\Gamma$ is depth. Let $x := (H^{l} )_{l=1}^{\Gamma}$, and denote DNN as $h_x(.):\mathbb{R}^{d_0}\rightarrow\mathbb{R}^{\Gamma}$, such that $h_x(a) := H^{\Gamma}\psi(H^{\Gamma-1}\cdots \psi(H^{1}a))$ for input feature $a \in \mathbb{R}^{d_0}$. To fit $h_x(.)$ to given dataset $(a_k, b_k)_{k=1}^N$ over a network where each agent $i$ owns $(a_{ij}, b_{ij})_{j=1}^{n}$, we need to solve Problem $\eqref{problem}$ with $f_i(x) = \frac{1}{n}\sum_{j=1}^{n}\tilde{\ell}(b_{ij}, h_x(a_{ij}))$ and $r\equiv 0$, where $\tilde{\ell}:\mathbb{R}^{d}\rightarrow \mathbb{R}_{+}$ is some loss function.  By \cite{lau2018proximal}, with $\tilde{\ell}$ being chosen as $\ell_2$ loss, $u$ is subanalytic and hence satisfies the KL property, with $\theta \in [0,1)$.

 \vspace{-0.2cm}
\subsection{The SONATA algorithm}
\label{subsec:alg}
 \noindent Our study leverages the SONATA algorithm~\cite{scutari2019distributed}  to solve Problem (P), which we briefly recall next. 
 
% SONATA is a typical gradient-tracking algorithm achieving success in many decentralized optimization problems. 
In SONATA, each agent $i$ maintains and updates iteratively a local copy ${x}_{i}\in\mathbb{R}^{d}$ of the global variable ${x}$, along with the auxiliary variable ${y}_{i}\in\mathbb{R}^{d}$ that represents the local estimate of    $\nabla f$. Denoting by ${x}_{i}^{\nu} \ (\text{resp.}  {y}_{i}^{\nu})$ the values of ${x}_{i} \ (\text{resp.} {y}_{i})$ at the iteration $\nu\in\mathbb{N}_{+}$, the update of each agent $i$ reads: given $x^\nu_i$, $y_i^\nu$, $i\in [m]$,
\begin{equation}
\label{eq:a}
\begin{aligned}
    %\label{eq:a1}
{x}_i^{\nu+1/2}& =\texttt{prox}_{\alpha r}(x^{\nu}_i-\alpha y_i^{\nu}),\\
 %\label{eq:a2}
    x^{\nu+1}_i&=\sum_{j=1}^m w_{ij}\, {x}_j^{\nu+1/2},\\ 
     y^{\nu+1}_i&=\sum_{j=1}^m w_{ij}\left(y_j^\nu+\nabla f_j(x_j^{\nu+1})-\nabla f_j(x_j^{\nu})\right).%\label{eq:a3}
\end{aligned}\end{equation}
Here, $\alpha>0$ is an appropriate constant stepsize, and the initialization is set as  $x_i^0\in\textit{dom}\ r$ and $y_i^0=\nabla f_i(x_i^0)$, $i\in [m]$.     The weights $w_{ij}$ are chosen according the following condition.  

%\begin{assumption}
   % [Gossip matrix]
    %\label{ass:W}
    %Given the graph $\mathcal{G}$ (assumed to be connected), the mixing matrix $W=[w_{ij}]_{i,j\in[m]}$ belongs to the class $\{W\in\mathbb{R}^{m\times m}: W=P_M(\overline{W})\}$ for some $M\in\mathbb{N}_+$ and $\overline{W}\in\mathbb{R}^{m\times m}_+$, where $P_M$ is a polynomial with degree at most $M$ satisfying $P_M(1)=1$, and $\overline{W}$ satisfies the following:\begin{itemize}
        %\item [(i)] $\overline{W}$ is compliant with $\mathcal{G}$, that is, $\overline{w}_{ii}>0$ for all $i\in[m]$, $\overline{w}_{ij}>0$ for all $(i,j)\in\mathcal{E}$ and $\overline{w}_{ij}=0$ otherwise;
        %\item[(ii)] $\overline{W}$ is doubly stochastic, namely $1^{\top}\overline{W}=1^{\top}$ and $\overline{W}1=1$. 
    %\end{itemize}
%\end{assumption}

\begin{assumption}
    [Gossip matrix]
    \label{ass:W}
    Given the   graph $\mathcal{G}$ (assumed to be connected), the mixing matrix ${W}:=({w}_{ij})_{i,j=1}^m$  satisfies the following:\begin{itemize}
        \item [(i)] ${w}_{ii}>0$, for all $i\in[m]$;\item[(ii)]${w}_{ij}>0$ for all $(i,j)\in\mathcal{E}$, and ${w}_{ij}=0$ otherwise; 
        \item[(iii)] ${W}$ is doubly stochastic, i.e., $1^{\top} {W}=1^{\top}$ and $ {W}1=1.$ \end{itemize}
\end{assumption}
Associated with Assumption~\ref{ass:W}, we define the following quantities used throughout the paper  \vspace{-0.1cm} \begin{equation}  \label{eq:rho}
               w_{\text{mx}}
            :=\sum_{i=1}^m \max_{j\in[m]}w_{ij}\quad \text{and}\quad \rho:=\|W-11^{\top}/m\|_2.
        \end{equation}
Notice that, under Assumption~\ref{ass:W}, it holds $\rho<1$. 

 Assumption~\ref{ass:W} is quite standard in the literature of decentralized algorithms; several rules have been proposed to generate such gossip matrices, see, e.g., ~\cite{xiao2005scheme,auzinger2011iterative,scaman2017optimal}.
 
 Finally, notice that, under Assumption~\ref{ass:W} (in particular,   $w_{ij}\neq0$ only if $(i,j)\in\mathcal{E}$), the SONATA algorithm  (\ref{eq:a}) is fully decentralized, as all the communications are performed only among neighboring agents.   \vspace{-0.2cm}
\subsection{Vector/matrix  representation}\label{subsec:ass&notation}

\noindent It is convenient to rewrite the agents' updates of the SONATA algorithm in vector/matrix form. To this end, we introduce the following notation. We stack the iterates and tracking variables into matrices, namely:  
$$X:=\left[  
x_1,x_2,\ldots,x_m\right]^\top, %\, {X}^{\nu+1/2}:=\begin{bmatrix}
%({x}_1^{\nu+1/2})^{\top}\\({x}_2^{\nu+1/2})^{\top}\\\vdots\\({x}_m^{\nu+1/2})^{\top}
%\end{bmatrix},\,
\quad Y:=\left[  
y_1, y_2,\ldots, y_m\right]^\top.$$
Accordingly, we define the pseudogradient
$$\nabla F(X):=\left[\nabla f_1(x_1),\nabla f_2(x_2),\cdots,\nabla f_m(x_m)\right]^{\top},$$ and the lifted functions
%$$\nabla\overline{ F}(X^{\nu}):=[\nabla f(x_1^{\nu}),\nabla f(x_2^{\nu}),\cdots, \nabla f(x_m^{\nu})]^{\top}.$$
%Finally, let 
$$U(X):=\sum_{i=1}^m u(x_i)\quad \text{and\quad }R(X):=\sum_{i=1}^m r(x_i). $$
At iteration $\nu$, the matrices above will take on the iteration index $\nu$ as a superscript, reflecting the corresponding iterates. 

 Using the above notation, we can rewrite the SONATA updates~\eqref{eq:a} %\textcolor{red}{(?)}
 in the following compact form: 
  \begin{subequations}
        %\tag{A1}
\begin{align}
\label{eq:A1}
{X}^{\nu+1/2}&=\texttt{prox}_{\alpha R}(X^{\nu}-\alpha Y^{\nu}),\\
%&=\textit{argmin}_{X} g(X)+\frac{1}{2\alpha}\|X^{\nu}-\alpha Y^{\nu}-X\|^2;
       \label{eq:A2}
            X^{\nu+1}&=  {W}  {X}^{\nu+1/2},\\
            \label{eq:A3}
    Y^{\nu+1}&= {W} \left(Y^{\nu}+\nabla F(X^{\nu+1})-\nabla F(X^{\nu}) \right).
    \end{align}
        \end{subequations}  
 Here, the proximal operator \texttt{prox} is applied row-wise.
 
%\begin{algorithm}[htb]
    %\caption{SONATA}
    %\label{alg:SONATA}
    %\begin{algorithmic}
        %\REQUIRE~~ $x_i^0\in \text{dom}(r)$, $y_i^0=\nabla f_i(x_i^0)$, $\alpha>0$, $R>0$, $\hat{W}\in\mathbb{R}^{m\times m}$.
        %\ENSURE~~ $\{X^{\nu}\}_{\nu\in\mathbb{N}}$
%\FOR{$\nu=1,2,\cdots,$}
 %       \STATE \begin{equation}
  %      \label{eq:A1}
   %     \tag{A1}
%\begin{aligned}
%X^{\nu+1/2}&=prox_{\alpha g}(X^{\nu}-\alpha Y^{\nu})\\
%&=\arg\min_{X} g(X)+\frac{1}{2\alpha}\|X^{\nu}-\alpha Y^{\nu}-X\|^2;
%\end{aligned}
 %       \end{equation}
       % \STATE \begin{equation}
        %\label{eq:A2}
        %\tag{A2}
         %   X^{\nu+1}= \hat{W}^{R}X^{\nu+1/2};
        %\end{equation}
        %\STATE \begin{equation}
        %\label{eq:A3}
        %\tag{A3}
   % Y^{\nu+1}=\hat{W}^{R}(Y^{\nu}+\nabla F(X^{\nu+1})-\nabla F(X^{\nu}) ).
        %\end{equation}
       % \ENDFOR
    %\end{algorithmic}
%\end{algorithm}

 Associated with \eqref{eq:A1}, we define the    direction 
 \begin{equation}\label{eq:def_D}
     D^\nu:= X^{\nu+1/2}-X^\nu,\end{equation}
 and the consensus   and tracking errors in matrix form:  
$$X_{\perp}^{\nu}:=(I-J)X^{\nu}, \quad {X}_{\perp}^{\nu+1/2}:=(I-J) {X}^{\nu+1/2}, $$  $$Y_{\perp}^{\nu}:=(I-J)Y^{\nu}  ,\quad \Delta^{\nu}:= [\nabla f(x_1^\nu),\ldots, \nabla f(x_m^\nu)]^\top -Y^\nu.$$  
  \vspace{-0.2cm}

\section{Preliminaries: Asymptotic Convergence}
\label{sec:asymptotic convergence}

This section investigates the asymptotic convergence of the SONATA algorithm~\eqref{eq:a} applied to Problem~\eqref{problem}. While previous analyses, such as that in \cite{scutari2019distributed}, have explored this topic, the Lyapunov functions used in those studies do not facilitate  the use of the KL property of the objective function to enhance convergence guarantees (instead, they require    postulating directly the KL of the Lyapunov  function). This limitation primarily arises because these merit functions are evaluated on the {\it average} of the agents' iterates. Due to consensus errors, it is challenging to ensure that if the average iterate falls within the region where the KL property holds for the objective function, each individual agent's iterate does as well.%\textcolor{blue}{ Without the latter property, the descent of the merit function is no longer guaranteed to be bounded away from zero even with the KL property, leading to the incapability of establishing fast convergence.}  

To address this challenge, this section introduces a novel approach that effectively leverages the KL property directly on the objective function. We propose a Lyapunov function that contains the sum of the objective function $u$   evaluated at each agent’s local variable, which permits to   
leverage the KL property of the objective function effectively.  
%As a result, a different approach is required to leverage the KL property of the objective function effectively, which is  the scope of this section. The proposed Lyapunov function is composed of the sum of the objective function $u$ evaluated at each agent's local variable.  
 
 We organize the proof as follows: \begin{itemize}
     \item {\bf Step 1:} We establish inexact decent of the average loss $U(X)=\sum_{i=1}^m u(x_i)$ along the agents' iterates generated by SONATA, subject to consensus and tracking errors;
     \item {\bf Step 2:} Leveraging bounds on  such  consensus and tracking errors as outlined in  \cite{SunDanScu19}, we  suitably merge objective and consensus dynamics into a novel  Lyapunov function that    is proved to descent along agents' iterates, yielding asymptotic convergence.   
 \end{itemize}
  All the derivations that follow   are obtained   postulating tacitly  {Assumptions~\ref{ass:function} and~\ref{ass:W}}.
  
\subsubsection{Step 1: inexact descent}
The core results to establish decent on the local agents' iterates rather than on their average, is     the  counterpart of the Jensen's inequality for nonconvex function  \cite[Thm. 1]{S04}, as recalled  below. 
\begin{lemma} 
    \label{lemma:nonconvex jensen}
    For any $L$-smooth function $u:\mathbb{R}^d\rightarrow \mathbb{R}$, set of weights $\{w_i\}_{i=1}^m$, with $w_i\geq 0$ and $\sum_{i=1}^m w_i=1$,    and   $x_i\in\mathbb{R}^d$, $i=1,\cdots,m$, the following holds
$$u\left(\sum_{i=1}^m w_i x_i\right)\leq\sum_{i=1}^m w_iu(x_i)+\frac{L}{2}\sum_{i=1}^m\sum_{j=1}^m\ w_i\,w_j\|x_j-x_i\|^2.$$
\end{lemma}

%\textcolor{cyan}{(Following corresponds to Lemma 10 of MP paper.)}
%\begin{lemma}
%\label{lemma:function inexact decent}
%Consider the problem (\ref{problem}) under Assumption~\ref{ass:function}. Let $\{X^{\nu}\}_{\nu\in \mathbb{N}}$ be the sequence generated by SONATA algorithm~\eqref{eq:A1}-\eqref{eq:A3} under Assumption~\ref{ass:W}, we have\begin{equation}
 %   \label{eq:global decent 2}
 %   \begin{aligned}
%U(X^{\nu+1})\leq& U(X^{\nu})-(\frac{1}{\alpha}-\frac{L}{2}-\frac{\xi}{2})\|D^{\nu}\|^2+\frac{1}{2\xi}\|\Delta^{\nu}\|^2\\
%&+4Lw_{mx}(\|X^{\nu}_{\perp}\|^2+\alpha^2\|Y^{\nu}_{\perp}\|^2).
%    \end{aligned}
%\end{equation}
%\end{lemma}
%\begin{proof}
Equipped with  Lemma~\ref{lemma:nonconvex jensen}, we proceed   studying descent of $U(X)$ along the iterates $X^\nu\to X^{\nu+1}$.  We have 
\begin{equation}
\label{eq:global decent 1}
\begin{aligned}U(X^{\nu+1})=&\sum_{i=1}^m u(\sum_{j=1}^m w_{ij}x_j^{\nu+1/2})\\\,\,\leq &  \underset{\texttt{term I}}{\underbrace{\sum_{i=1}^m\sum_{j=1}^m w_{ij}u(x_j^{\nu+1/2})}}+\frac{L}{2}\sum_{i=1}^m\sum_{k=1}^m\sum_{l=1}^m w_{ik}w_{il} \underset{\texttt{term II}}{\underbrace{\|x_k^{\nu+1/2}-x_l^{\nu+1/2}\|^2}}
+\sum_{i=1}^m\underset{\texttt{term III}}{\underbrace{\left(r(x_i^{\nu+1})-\sum_{j=1}^m w_{ij}r(x_j^{\nu+1/2})\right)}}. 
  %  \leq& U(X^{\nu+1/2})+4Lw_{mx}(\|X^{\nu}_{\perp}\|^2+\alpha^2\|Y^{\nu}_{\perp}\|^2),
  \end{aligned}\end{equation}
We proceed bounding the above terms separately.  %Assumption~\ref{ass:W}, we notice that $\texttt{term I}=U(X^{\nu+1/2})$. Invoking the $L$-smoothness of $f_i$ and convexity of $r$, for any given $\xi>0$, it can be upper bounded as
\begin{equation}
\label{eq:term I}\begin{aligned}
    \texttt{term I}&=U(X^{\nu+1/2})\leq U(X^{\nu})-(\frac{1}{\alpha}-\frac{L}{2}-\frac{\xi}{2})\|D^{\nu}\|^2+\frac{1}{2\xi}\|\Delta^{\nu}\|^2,
\end{aligned}\end{equation}
 for any given    $\xi>0$. Here, the equality follows from  Assumption~\ref{ass:W}, and in the inequality we used the smoothness of $f_i$ in conjunction with   the convexity of $r$ and the Young's inequality.
 
As far as \texttt{term II} is concerned,  we have 
$$\begin{aligned}
    \texttt{term II}\overset{\eqref{eq:A1}}{=}&\|\texttt{prox}(x_k^{\nu}-\alpha y^{\nu}_k)-\texttt{prox}(x_l^{\nu}-\alpha y_l^{\nu})\|^2
    \overset{(b)}{\leq}  \|x_k^{\nu}-x_l^{\nu}-\alpha(y_k^{\nu}-y_l^{\nu})\|^2.%\\\leq& 4\left[\|x_{k,\perp}^{\nu}\|^2+\|x_{l,\perp}^{\nu}\|^2+\alpha^2(\|y_{k,\perp}^{\nu}\|^2+\|y_{l,\perp}^{\nu}\|^2)\right].
\end{aligned}$$
where  (b) follows from the non-expensiveness of the proximal operator.

Finally, using the convexity of $r$ and invoking the  Jensen's inequality, we deduce 
$\texttt{term III} \leq 0.$

Using  the above bounds in \eqref{eq:global decent 1}, we obtain
\begin{equation} \label{eq:global decent 2}
    \begin{aligned}
U(X^{\nu+1})
\leq& U(X^{\nu})-\left(\frac{1}{\alpha}-\frac{L}{2}-\frac{\xi}{2}\right)\|D^{\nu}\|^2+\frac{1}{2\xi}\|\Delta^{\nu}\|^2+4Lw_{\text{mx}}(\|X^{\nu}_{\perp}\|^2+\alpha^2\|Y^{\nu}_{\perp}\|^2). 
    \end{aligned}
\end{equation}

\subsubsection{Step 2: Lyapunov function and its descent} Using  \cite[Lemma 3.3]{SunDanScu19},   the tracking error     $\|\Delta^{\nu}\|$ in \eqref{eq:global decent 2} is bounded as   
\begin{equation}
    \label{eq:tr dynm}
            \|\Delta^{\nu}\|^2 
      %  =&\sum_{i=1}^m\|\frac{1}{m}\sum_{j=1}^m(\nabla f_j(x_{i}^{\nu})-\nabla f_j(x_\perp))-y_{i,\perp}^{\nu}\|^2\\
      %  \leq 2\|Y^{\nu}_{\perp}\|^2+\frac{2L_{\text{mx}}^2}{m}\sum_{i=1}^m\sum_{j=1}^m\|x_i^{\nu}-x_\perp\|^2 \\
        %\leq& 2\|Y^{\nu}_{\perp}\|^2+\frac{4L_{mx}^2}{m}\sum_{i=1}^m\|x_{i,\perp}^{\nu}\|^2+\sum_{j=1}^m \|x_{j,\perp}^{\nu}\|^2\\
        \leq 2\|Y^{\nu}_{\perp}\|^2+4 L_{\text{mx}}^2\|X^{\nu}_{\perp}\|^2.
\end{equation}
 
Substituting this bound in~\eqref{eq:global decent 2}, yields
\begin{equation} 
\label{eq:global decent 3}\begin{aligned}
&U(X^{\nu+1})
\leq U(X^{\nu})-\left(\frac{1}{\alpha}-\frac{L}{2}-\frac{\xi}{2}\right)\|D^{\nu}\|^2+\left(4Lw_{\text{mx}}+\frac{2L_{\text{mx}}^2}{\xi}\right)\|X^{\nu}_{\perp}\|^2+\left(4Lw_{\text{mx}}\alpha^2+\frac{1}{\xi}\right)\|Y^{\nu}_{\perp}\|^2. 
    \end{aligned} \end{equation}
Adjusting \cite[Prop. 3.5]{SunDanScu19}   to the algorithm update (\ref{eq:A1}),   the  consensus dynamics $\|X^{\nu}_{\perp}\|$ and $\|Y^{\nu}_{\perp}\|$ read   
\begin{equation}
    \label{eq:iter dynm}
    \begin{aligned}
        \|X^{\nu+1}_{\perp}\|\leq  &\rho \|X_{\perp}^{\nu}\|+\rho \|D^{\nu}\|,\medskip \\
         \|Y_{\perp}^{\nu+1}\|
        \leq & \rho \|Y_{\perp}^{\nu}\|+2\rho L_{\text{mx}}\|X_{\perp}^{\nu}\|+\rho L_{\text{mx}}\|D^{\nu}\|.
    \end{aligned}
\end{equation}

We proceed bounding the positive term on the RHS of \eqref{eq:global decent 3}. Since such a term is a linear combination of   $\|X_{\perp}^{\nu}\|^2$ and $\|Y_{\perp}^{\nu}\|^2$,  we can upper bound it using  
$$\mathcal{E}^\nu:=c_1\|Y_{\perp}^{\nu}\|^2+c_2\|X_{\perp}^{\nu}\|^2,$$
where $c_1$ and $c_2$ are  positive coefficients offering some degrees of freedom. We can thus bound \eqref{eq:global decent 3} as   \begin{equation}
\begin{aligned}
\label{eq:global decent 4}
        U(X^{\nu+1})\leq U(X^{\nu})-\left(\frac{1}{\alpha}-\frac{L}{2}-\frac{\xi}{2}\right)\|D^{\nu}\|^2+\tilde{c}_1\,\mathcal{E}^\nu,
    \end{aligned}
    \end{equation} where \vspace{-.1cm}
  $$\tilde{c}_1:=\max\left\{\frac{4Lw_{\text{mx}}}{c_2}+\frac{2L_{\text{mx}}^2}{c_2\xi},\frac{1}{c_1\xi}+\frac{4Lw_{\text{mx}}\alpha^2}{c_1}\right\}.$$ 
 
Using~\eqref{eq:iter dynm}, it is not difficult to check that the dynamics of $\mathcal{E}^{\nu}$   along the trajectory of the algorithm satisfy
\begin{equation}
\label{eq:cover dynm}
  \begin{aligned}
\mathcal{E}^{\nu+1}
        \leq& \rho^2\max\left\{3,2+6\frac{c_1}{c_2}  L_{\text{mx}}^2\right\}\mathcal{E}^\nu+(3\rho^2L_{\text{mx}}^2c_1+2\rho^2c_2)\|D^{\nu}\|^2.
    \end{aligned}\end{equation}
    Notice that, for sufficiently small $\rho$, $\mathcal{E}^{\nu}$  contracts up to a perturbation proportional to $\|D^{\nu}\|^2$. 
  For the sake of simplicity, we set   $c_1=2$ and $c_2=4L_{\text{mx}}^2$, to minimize the   contraction coefficient (albeit this choice might not be optimal overall).  This yields to \vspace{-.1cm}
\begin{equation}
\label{tc1}\tilde{c}_1=\frac{1}{2\xi}+\max\left\{\frac{L}{L_{\text{mx}}^2}w_{\text{mx}},2L\alpha^2w_{\text{mx}}\right\}.\end{equation}
 
Chaining  (\ref{eq:global decent 4}) and  (\ref{eq:cover dynm}) (with the latter weighted by a positive constant $\gamma$, to be determined),  we obtain      \begin{equation}
        \label{eq:lyapunov decent}
      \mathcal{L}^{\nu+1}  \leq \mathcal{L}^{\nu}-\tilde{c}_2\|D^{\nu}\|^2-\tilde{c}_3\mathcal{E}^{\nu},
\end{equation}
where $\mathcal{L}^\nu$ %\textcolor{red}{I do not like much the noation, technically this is not a function of $X$ that we can define independently of the iteration index, unless we define $\mathcal E^\nu$ as a function as well. let's talk later} \textcolor{blue}{How about denoting as $\mathcal{L}^{\nu}$, can we still call it a lyapunov "function" in the text?} 
is the candidate Lyapunov function at iteration $\nu$ along the trajectory of the algorithm,  defined as
\begin{equation}\label{eq:Lyapunov_function}
\mathcal{L}^{\nu}:=U(X^{\nu})+\gamma\mathcal{E}^\nu,\end{equation}
with 
$\gamma>0$ being  a free parameter to properly choose, and 
  \begin{subequations}
\begin{align}
\label{tc2}
\tilde{c}_2 :=&\frac{1}{\alpha}-\frac{L}{2}-\frac{\xi}{2}-14L_{\text{mx}}^2\gamma\,\rho^{2},\\
       \label{tc3}
\tilde{c}_3 :=&(1-5\rho^{2})\gamma-\tilde{c}_1.
    \end{align}
        \end{subequations}
Under  $\tilde{c}_2, \tilde{c}_3>0$, it follows from (\ref{eq:lyapunov decent}) and Assumption~\ref{ass:function} that $\mathcal{L}^{\nu}\to \mathcal{L}^{\infty}>-\infty$, as $\nu\to \infty$.  %is a decreasing sequence. By Assumption~\ref{ass:function}, we know that $\{\mathcal{L}^{\nu}\}$ is lower bounded. Therefore, $\{\mathcal{L}^{\nu}\}$ converges to a limit $\mathcal{L}^{\infty}>-\infty$. Rewrite \eqref{eq:lyapunov decent} in the following form
%$$\begin{aligned}0\leq \gamma_1\|D^{\nu}\|^2+\gamma_2\mathcal{E}^\nu\leq\mathcal{L}^{\nu}-\mathcal{L}^{\nu+1},\end{aligned}$$
%and take the limit on both hands side, we have by squeeze theorem that as $\nu\rightarrow\infty$,
Hence, $$\|D^{\nu}\|, \, \|X_{\perp}^{\nu}\|, \, \|Y_{\perp}^{\nu}\|,\,\|\Delta^{\nu}\|\rightarrow 0,\quad  U(X^{\nu})\rightarrow \mathcal{L}^{\infty}.%\|\Delta^{\nu}\|\rightarrow 0.
$$
%Therefore, we have $$\lim_{\nu\rightarrow\infty}U(X^{\nu})= \mathcal{L}^{\infty}.$$ 

Further,    $\{U(X^{\nu+1/2})\}$ also  converges to $\mathcal{L}^{\infty}$, as showed next. Using~\eqref{eq:global decent 1}, we have
\begin{equation}
\label{eq:term I'}
\begin{aligned}
    %U(X^{\nu+1/2})\leq &U(X^{\nu})-\left(\frac{1}{\alpha}-\frac{L}{2}-\frac{\xi}{2}\right)\|D^{\nu}\|^2+\frac{1}{2\xi}\|\Delta^{\nu}\|^2.\\
    U(X^{\nu+1/2})\geq& U(X^{\nu+1})-4Lw_{\text{mx}}(\|X_{\perp}^{\nu}\|^2+\alpha^2\|Y_{\perp}^{\nu}\|^2);
\end{aligned}    
\end{equation}
%where the first inequality uses the bound in \texttt{term I} and the second inequality comes from substituting the bound for \texttt{term II} and \texttt{term III} in ~\eqref{eq:global decent 1} while keeping $\texttt{term I}= U(X^{\nu+1/2})$.
%Noticed that $\{U(X^{\nu+1/2})\}$ is bounded, we can 
taking the  limsup and liminf of~\eqref{eq:term I} and~\eqref{eq:term I'},   respectively, yields
$$\mathcal{L}^{\infty}\leq \liminf_{\nu} U(X^{\nu+1/2})\leq \limsup_{\nu} U(X^{\nu+1/2})\leq \mathcal{L}^{\infty}.$$
%Therefore, it can be concluded that,
%$$\lim_{\nu\rightarrow\infty} U(X^{\nu+1/2})=\lim_{\nu\rightarrow\infty}U(X^{\nu})=\mathcal{L}^{\infty}.$$

%\subsubsection{Step 3: From functions to the iterates}
%Given the asymptotic convergence of merit functions, it is expected that we can investigate the asymptotic behaviour on the iterates, which is of our particular interest. Notice that with the absence of further curvature information of our merit function $U$, the convergence of $\{X^{\nu}\}$ is not guaranteed. Nevertheless, good property can be investigated on the possible limit points, which may be useful after leveraging the KL property. 
We conclude the proof establishing properties of the limit points of $\{X^\nu\}$. Let  $X^{\infty}$ be an accumulation point of  $\{X^{\nu}\}_{\nu\in\mathbb{N}}$, such that %  there exists a subsequence $\{X^{\nu_k}\}_{k=1}^{\infty}$ such that
$\lim_{k\rightarrow\infty}X^{\nu_k}=X^{\infty},$ where $\{\nu_k\}_k\subseteq \mathbb N$ is a suitable subsequence. It must be
%Since $F$ is continuously differentiable by Assumption~\ref{ass:function}. We have
\begin{equation}\label{eq:consensus_limit}
    X^{\infty}=1 (x^*)^{\top},\quad \text{for some } x^*\in \mathbb{R}^d. 
\end{equation} 
Further, using  still  $\{\nu_k\}_k$   without loss of generality (w.l.o.g), we have 
$$\begin{aligned}
Y^\infty:=&\lim_{k\rightarrow\infty}Y^{\nu_k}=\lim_{k\rightarrow\infty}[\nabla f(x_1^{\nu_k}),\ldots, \nabla f(x_m^{\nu_k})]^\top-\Delta^{\nu_k}=1 \nabla f(x^*)^{\top},
\end{aligned}$$
$$\lim_{k\rightarrow\infty}X^{\nu_k+1/2}\overset{\eqref{eq:def_D}}{=}\lim_{k\rightarrow\infty}X^{\nu_k}+D^{\nu_k}=X^{\infty}.$$ 
Invoking the continuity of the prox operator, we deduce  
$$\begin{aligned}X^{\infty}=&\lim_{k\rightarrow\infty}X^{\nu_k+1/2}\overset{\eqref{eq:A1}}{=}\lim_{k\rightarrow\infty}\texttt{prox}_{\alpha R}(X^{\nu_k}-\alpha Y^{\nu_k})
=\texttt{prox}_{\alpha R}(1(x^*)^\top-\alpha\cdot 1\left(\nabla f(x^*)\right)^{\top}),
\end{aligned}$$
which, together with \eqref{eq:consensus_limit}, implies $0\in\partial u(x^*)$.   %Since $\|X_{\perp}^{\nu}\|\rightarrow 0$, it must be 
%$$\begin{aligned}0=&\lim_{k\rightarrow\infty}\|X_{\perp}^{\nu_k}\|=\lim_{k\rightarrow\infty}\|(I-J)X^{\nu_k}\|=\|X^{\infty}-JX^{\infty}\|\\
%&\Rightarrow X^{\infty}=JX^{\infty}=1\cdot \frac{1^{\top}}{m}X^{\infty}.\end{aligned}$$
%Thus $X^{\infty}=1 (x^*)^{\top}$ for $x^*=(\frac{1^{\top}}{m}X^{\infty})^{\top}\in\mathbb{R}^d$. 
Therefore %$0\in\partial u(x^*)$, which means 
$x^*$ is a stationary point of $u$ in~\eqref{problem}.

 We summarize the above results in the following theorem, where we conveniently chose the free parameters to satisfy the required conditions $\tilde{c}_2,\tilde{c}_3>0$.  
\begin{theorem}
\label{thm:1}
    Consider Problem~(\ref{problem}) under Assumption~\ref{ass:function}. Let $\{X^{\nu}\}_{\nu\in \mathbb{N}}$ be the sequence generated by the SONATA algorithm~\eqref{eq:A1}-\eqref{eq:A3}, under Assumption~\ref{ass:W}. 
     Suppose  
     
         {\bf (i)} $\rho< {1}/{\sqrt{5}}$, and
         
          {\bf (ii)} the free parameters $\alpha,\gamma,\xi>0$    %\textcolor{red}{yes, if you never use a different value of $\xi$, you can set already here L}
          are chosen such that $$\gamma>\frac{1/(2\xi)+Lw_{\text{mx}}/L_{\text{mx}}^2}{1-5\rho^2},$$ $$\alpha<\min\left\{\frac{1}{L/2+\xi/2+14L_{\text{mx}}^2\gamma\rho^2},\sqrt{\frac{(1-5\rho^2)\gamma-1/(2\xi)}{2Lw_{\text{mx}}}}\right\}.$$  
     Then, % then with the choice of stepsize 
    %$$0<\alpha< \min\left\{\frac{1}{\sqrt{2}L},\frac{1}{14(w_{\text{mx}}+1)\cdot\frac{L_{\text{mx}}^2}{L}\cdot\frac{\rho^2}{1-5\rho^2}+L}\right\},$$ 
  %  \textcolor{blue}{Another version: then with the choice of $\alpha,\gamma,\xi>0$ such that
  %  $$\alpha<\min\left\{\frac{1}{L/2+\xi/2+14L_{\text{mx}}^2\gamma\rho^2},\sqrt{\frac{(1-5\rho^2)\gamma-1/(2\xi)}{2Lw_{\text{mx}}}}\right\},$$
%$$\gamma>\frac{1/(2\xi)+Lw_{\text{mx}}/L_{\text{mx}}^2}{1-5\rho^2,}$$}
   as $\nu\to \infty$. 
    $$ \|D^{\nu}\|,\,\, \|X_{\perp}^{\nu}\|, \,\,\|Y_{\perp}^{\nu}\|, \,\,\|\Delta^{\nu}\|\to 0,$$ and
     $$U(X^{\nu}),\, U(X^{\nu+1/2})\to \mathcal{L}^{\infty},$$ for some $\mathcal{L}^{\infty}\in\mathbb{R}$.  Furthermore, every  accumulation point ${X}^{\infty}$ of $\{X^{\nu}\}_{\nu\in\mathbb{N}}$  is of the form  $X^{\infty}=1({x}^*)^{\top}$, with  ${x}^*$ being  a stationary solution of (\ref{problem}).
\end{theorem}
 While the theorem certifies stationarity of  every limit point of the agents' iterates   $\{x_i^\nu\}$, convergence of the whole sequences is not guaranteed.  The next section addresses this issue, under the KL property of $u$. 
 
\section{Convergence Rate under the KL Property}
\label{sec:rate}
 \noindent %Equipped with the  asymptotic convergence results established in 
 Building on Theorem~\ref{thm:1}, we  proceed proving   convergence of the sequence $\{X^\nu\}$ and its convergence rate, under the   KL property of $u$. 
Through the   section, we postulate the  setting of Theorem~\ref{thm:1} and  { the KL property of $u$ at its stationary points. }

Since $\mathcal{L}^{\nu}\to \mathcal{L}^{\infty}$ as $\nu\to \infty$, it is convenient to rewrite     \eqref{eq:lyapunov decent} in terms of the offset   quantity $\Delta\mathcal{L}^{\nu}:=\mathcal{L}^{\nu}-\mathcal{L}^{\infty}$, that is,   
\begin{equation}
\label{eq:lyapunov gap dyn}
  \left(\mathcal{T}^{\nu}\right)^2\leq \Delta \mathcal{L}^{\nu}- \Delta \mathcal{L}^{\nu+1},
\end{equation} where $$\mathcal{T}^{\nu}:=\sqrt{\tilde{c}_2\|D^{\nu}\|^2+\tilde{c}_3\mathcal{E}^{\nu}}.$$
Notice that $\mathcal{T}^{\nu},\Delta\mathcal{L}^{\nu}\rightarrow 0$, as $\nu\to \infty$ (by Theorem~\ref{thm:1}).

%\subsubsection{Capture the limiting behavior of the iterates.}
%Considering the limiting behavior of $\{X^{\nu}\}$, it can be either bounded or unbounded. Let's assume for now the nontrivial case that it is bounded. Then with conclusion in Theorem~\ref{thm:1}, $\{X^{\nu}\}$ admits an accumulation point $X^{\infty}=1(x^*)^T$ with $x^*$ being a stationary point of Problem~\eqref{problem} and $D^\nu\to0$. By facts that $X^{\nu+1/2}=X^{\nu}+D^{\nu}$ , we conclude that $\{X^{\nu+1/2}\}$ is also bounded and $X^{\infty}$ is also an accumulation point of it. 
\subsection{Sequence convergence}To prove convergence of the whole sequence, we show next that   $\|X^{\nu+1}-X^{\nu}\|$ is summable. 

By~\eqref{eq:a} and the non-expansiveness of the proximal operator,   
\begin{equation}
    \label{eq:X-E}
    \begin{aligned}
\|X^{\nu+1}-X^{\nu}\|
\leq & \sqrt{2\|W-I\|^2\|X_{\perp}^{\nu}\|^2+2\|W\|^2\|D^{\nu}\|^2}\\
\leq&\sqrt{2\|W\|^2\|D^{\nu}\|^2+\frac{\|W-I\|^2}{2L_{\text{mx}}^2}\mathcal{E}^{\nu}}
 \\
 \leq&c_3\mathcal{T}^{\nu},
    \end{aligned}
\end{equation}
  where $$c_3:=\sqrt{\max\left\{\frac{\|W-I\|^2}{2\tilde{c}_3L_{\text{mx}}^2},\frac{2\|W\|^2}{\tilde{c}_2}\right\}}.$$

Therefore, it is sufficient to to show that   $\{\mathcal{T}^{\nu}\}$ is summable.

Using~\eqref{eq:lyapunov gap dyn} along with   (see Lemma~\ref{lemma:MVT} in the  Appendix~\ref{subsec:A})   $$a-b\leq \frac{1}{1-\theta}a^{\theta}(a^{1-\theta}-b^{1-\theta}),$$ with $a=\Delta \mathcal{L}^{\nu}$ and $b=\Delta \mathcal{L}^{\nu+1}$, we obtain %we  $\Delta\mathcal{L}^{\nu}-\Delta\mathcal{L}^{\nu+1}$  in (the first inequality in)~\eqref{eq:G bound} into $\Delta\mathcal{L}^{\nu}$ and the summable term  $\Delta\mathcal{L}_{\theta}^{\nu}$: 
\begin{equation}
    \label{eq:Enu dynm}
\begin{aligned}
    (\mathcal{T}^{\nu})^2
    \leq&
    \frac{1}{1-\theta}(\Delta \mathcal{L}^{\nu})^{\theta}\underbrace{\left[(\Delta \mathcal{L}^{\nu})^{1-\theta}-(\Delta \mathcal{L}^{\nu+1})^{1-\theta}\right]}_{\Delta\mathcal{L}^{\nu}_{\theta}}.
\end{aligned}
\end{equation}
The challenge is now establishing   an     upper bound of $\Delta\mathcal{L}^{\nu}$ in terms of $\mathcal{T}^{\nu}$. The classical path followed in the literature would call for some  growth property of $\Delta\mathcal{L}^{\nu}$. However, $\Delta \mathcal{L}^{\nu}$, as function of the iterates,  does not inherit the KL property of $u$; hence, one cannot postulate any  growth property for  
 $\Delta \mathcal{L}^{\nu}$. The proposed, novel  approach is to  leverage  instead directly the KL   property of $u$ while using the    asymptotic  convergence of $\Delta \mathcal{L}^{\nu+1}$ as established in \eqref{eq:lyapunov gap dyn}. Specifically,      we first decompose $\Delta \mathcal{L}^{\nu+1}$  as \begin{equation}
\begin{aligned}
\label{eq:lyapunov enu}
\Delta\mathcal{L}^{\nu+1}=&U(X^{\nu+1})-\mathcal{L}^{\infty}+\gamma\mathcal{E}^{\nu+1}
    \overset{\eqref{eq:global decent 1}, \eqref{eq:cover dynm} }{\leq}    U(X^{\nu+1/2})-\mathcal{L}^{\infty}+\tilde{c}_4(\mathcal{T}^{\nu})^2,
\end{aligned}
\end{equation}
where   
\begin{equation}
    \label{tc4}
    \tilde{c}_4:=\max\left\{\frac{14\gamma\rho^{2}L_{\text{mx}}^2}{\tilde{c}_2},\frac{5\gamma\rho^2+\tilde{c}_1-1/\xi}{\tilde{c}_3}\right\}.
\end{equation}
Then, invoking the KL property of $u$, we can    lower bound  $\mathcal T^\nu$ in terms of  $ U(X^{\nu+1/2})-\mathcal{L}^{\infty}$, see Lemma~\ref{lemma:U KL} below.  This result along with 
 \eqref{eq:lyapunov enu} provides the desired upper bound of $\Delta \mathcal{L}^{\nu+1}$ in terms of  $\mathcal T^\nu$, which can be used in (\ref{eq:Enu dynm}) to establish  the summability of $\mathcal{T}^{\nu}$. 
\begin{lemma}
    \label{lemma:U KL}
      %Consider Problem~(\ref{problem}) under Assucanmption~\ref{ass:function} and we further assume that the objective function $u$ satisfies Assumption~\ref{ass:KL} with exponent $\theta\in[0,1)$. Let $\{X^{\nu}\}_{\nu\in \mathbb{N}}$ be the sequence generated by the SONATA algorithm~\eqref{eq:A1}-\eqref{eq:A3}, under Assumption~\ref{ass:W} with the same parameter settings as Theorem~\ref{thm:1}. 
      
      Inherit  the setting of Theorem~\ref{thm:1}. Let $X^\infty=1 (x^\star)^\top$ be an accumulation point of $\{X^{\nu}\}$, where $x^\star$ is   some  critical point of   $u$. Further assume that $u$ satisfies the KL property at $x^\star$, with exponent $\theta\in[0,1)$ and parameters $\kappa,\eta\in(0,\infty)$. Then, there exists a neighborhood $\mathcal{N}_{\infty}$ of $X^{\infty}$ and $\nu_1\geq 0$ such that {\bf (i)} the set
      $$V:=\left\{\nu\geq\nu_1:X^{\nu+1/2}\in\mathcal{N}_{\infty}\right\}\neq\emptyset;$$  {\bf (ii)} for each $\nu\in V$, $$ U(X^{\nu+1/2})-\mathcal{L}^{\infty}<1,\quad  0\leq \tilde{c}_5\mathcal{T}^{\nu}<1,$$
and, for $\theta\in(0,1)$,
\begin{equation*}
    %\label{eq:U KL}
\begin{aligned}
  U(X^{\nu+1/2})-\mathcal{L}^{\infty}< \kappa^{-\frac{1}{\theta}}\left(\tilde{c}_5\mathcal{T}^{\nu}\right)^{\frac{1}{\theta}},
\end{aligned}
\end{equation*}
where
% $\kappa$ is the KL parameter associated with  the stationary point ${x}^*$, and
\begin{equation}
\label{tc5}
    \begin{aligned}
\tilde{c}_5:=\max\left\{d^{\theta-\frac{1}{2}},1\right\}\cdot\sqrt{\max\left\{\frac{3\left(L^2+\frac{1}{\alpha^2}\right)}{\tilde{c}_2},\frac{3}{\tilde{c}_3}\right\}}.\end{aligned}
\end{equation}
\end{lemma}
\begin{proof}
    See Sec.~\ref{subsec:lemma proof}.
\end{proof}
%Lemma~\ref{lemma:U KL} indicates that in some neighborhood $\mathcal{N}$ of $X^{\infty}$,~\eqref{eq:U KL} holds. 
In the setting of  Lemma~\ref{lemma:U KL}, we may assume, without loss of generality, that $|V|=\infty$. Using 
\eqref{eq:lyapunov enu}, for $\theta\in(0,1)$, yields
\begin{equation}
    \label{eq:lyapunov Enu}
    \begin{aligned}
 \Delta \mathcal{L}^{\nu+1}
\leq\kappa^{-\frac{1}{\theta}}\left(\tilde{c}_5\mathcal{T}^{\nu}\right)^{\frac{1}{\theta}}+\tilde{c}_4(\mathcal{T}^{\nu})^2,\quad\forall\nu\in V;
\end{aligned}
\end{equation}

We proceed based upon the  values of the KL exponent $\theta$. %We shall derive the upper bound of $\mathcal{T}^{\nu}$ using~\eqref{eq:lyapunov Enu} more carefully case by case.

\indent$\bullet$ $\textbf{Case 1: }\theta\in(0,1/2]$. Using~\eqref{eq:lyapunov Enu} and  ${1}/{\theta}\geq 2$, yields 
\begin{equation}
    \label{eq: Lyapunov KL decent}
    \Delta \mathcal{L}^{\nu+1}\leq \tilde{c}_6(\mathcal{T}^{\nu})^2,\quad\forall\nu\in V,
\end{equation}
where\begin{equation}
\label{tc6}
\tilde{c}_6=\left(\Tilde{c}_5^2/\kappa^{\frac{1}{\theta}}+\tilde{c}_4\right).
\end{equation}
Chaining~\eqref{eq: Lyapunov KL decent} and~\eqref{eq:lyapunov gap dyn}, specifically    $\eqref{eq:lyapunov gap dyn}+\omega\times~\eqref{eq: Lyapunov KL decent}$, with a sufficiently small  constant $\omega>0$, we obtain
$$\begin{aligned}
(1+\omega)\Delta \mathcal{L}^{\nu+1}\leq \Delta \mathcal{L}^{\nu} -\left(1-\omega\tilde{c}_6\right)(\mathcal{T}^{\nu})^2,\quad  \forall \nu\in V.
\end{aligned}$$
Choosing   
$0<\omega\leq1/\tilde{c}_6$ ensures contraction of $\Delta \mathcal{L}^{\nu}$: 
\begin{equation}
    \label{eq:linear Lyapunov}
    \begin{aligned}
   \Delta\mathcal{L}^{\nu+1}\leq\frac{1}{1+\omega}\Delta \mathcal{L}^{\nu}\leq c_4\left(\frac{1}{1+\omega}\right)^{\nu+1}
\end{aligned},\quad \forall\nu\in V,
\end{equation} and some    $c_4>0$. 
 To optimize the contraction factor, we set 
\begin{equation}
    \label{omega}
    \begin{aligned}
    \omega:=&\frac{1}{\tilde{c}_6}
    =\frac{1}{\max\left\{\frac{3(L^2+1/\alpha^2)}{\kappa^{\frac{1}{\theta}}\tilde{c}_2},\frac{3}{\kappa^{\frac{1}{\theta}}\tilde{c}_3}\right\}+\max\left\{\frac{14\gamma\rho^2L^2_{\text{mx}}}{\tilde{c}_2},\frac{5\gamma\rho^2+\tilde{c}_1-\frac{1}{\xi}}{\tilde{c}_3}\right\}},
    \end{aligned}
\end{equation}
where $\tilde{c}_1,\tilde{c}_2,\tilde{c}_3$ are defined in~\eqref{tc1},~\eqref{tc2} and~\eqref{tc3}, respectively.
Using~\eqref{eq:linear Lyapunov} in~\eqref{eq:lyapunov gap dyn} yields\begin{equation}
\label{eq:contract 1}
    \mathcal{T}^{\nu}\leq c_5\tau^{\nu}, \quad \forall \nu\in V,\quad \text{with}\quad \tau=\sqrt{\frac{1}{1+\omega}},\,\, c_5=\sqrt{c_4}>0. 
\end{equation}

\indent$\bullet$ $\textbf{Case 2: }\theta\in(1/2,1)$.  From~\eqref{eq:lyapunov Enu} and  ${1}/{\theta}<2$, we deduce
\begin{equation}
\label{eq:lyapunov Enu2}
    \Delta \mathcal{L}^{\nu+1}\leq \tilde{c}_7(\mathcal{T}^{\nu})^{\frac{1}{\theta}},\quad \forall\nu\in V,
\end{equation}
where\begin{equation}
\label{tc7}
\tilde{c}_7=\left((\tilde{c}_5/\kappa)^{\frac{1}{\theta}}+\tilde{c}_4\right).
\end{equation}
 Using~\eqref{eq:lyapunov Enu2} in~\eqref{eq:Enu dynm} yields
\begin{equation}
    \label{eq:contract 2}
\begin{aligned}
     \mathcal{T}^{\nu+1}\leq &\frac{1}{2}\mathcal{T}^{\nu}+c_6\Delta\mathcal{L}^{\nu+1}_{\theta},\quad \forall\nu\in V,
\end{aligned}
\end{equation}
where $c_6=(2-2\theta)^{-1}\tilde{c}_7>0$.

\indent$\bullet$ $\textbf{Case 3: }\theta=0$. We consider the following two  sub-cases.
\begin{itemize}
    \item[\bf (i)]   There exists $\nu_2\in V$ such that for all $\nu\geq\nu_2$, $\nu\in V$, $U(X^{\nu+1/2})-\mathcal{L}^{\infty}< 0$. By ~\eqref{eq:lyapunov enu}, we have 
\begin{equation}
    \label{eq:lyapunov Enu3} \Delta\mathcal{L}^{\nu+1}\leq\tilde{c}_4(\mathcal{T}^{\nu})^2,\quad \forall\nu\geq\nu_2,\,\nu\in V.
\end{equation}
Following the analysis as in Case 1,   we conclude 
\begin{equation}
    \label{eq:contract 3.1}
    \mathcal{T}^{\nu}\leq c_5(\tau')^{\nu},\quad \tau':=\sqrt{\frac{1}{1+\omega'}},\quad\forall\nu\geq\nu_2,\nu\in V,
\end{equation}
where
\begin{equation}
    \label{omega'}
    \begin{aligned}
    \omega':=\frac{1}{\tilde{c}_4}
    =\min\left\{\frac{\tilde{c}_2}{14\gamma\rho^2L^2_{\text{mx}}},\frac{\tilde{c}_3}{5\gamma\rho^2+\tilde{c}_1-1/\xi}\right\},
    \end{aligned}
\end{equation}
ans $\tilde{c}_1,\tilde{c}_2,\tilde{c}_3$ are defined in~\eqref{tc1},~\eqref{tc2} and~\eqref{tc3}, respectively.
\item[\bf (ii)]  For all $\nu\in V$, there exists $\nu'\geq\nu$, $\nu'\in V$ such that $U(X^{\nu'+1/2})-\mathcal{L}^{\infty}\geq0$.  Pick such $\nu'$ and denote
$$V':=\{\nu'\in V: U(X^{\nu'+1/2})-\mathcal{L}^{\infty}\geq0\}.$$
Then, in setting of Lemma~\ref{lemma:U KL}, we have that  $0\leq U(X^{\nu'+1/2})-\mathcal{L}^{\infty}<1$, for all  $\nu'\in V'$. By  Lemma~\ref{lemma:G-D} (See Appendix A), we have that for any $\nu'\in V'$ there exists $G^{\nu'}\in\mathbb{R}^{m\times d}$, $G^{\nu'}\in\partial U(X^{\nu'+1/2})$,  such that 
\begin{equation}
    \label{eq:G-D'}
    \|G^{\nu'}\|^2
        \leq3\left(L^2+\frac{1}{\alpha^2}\right)\|D^{\nu'}\|^2+3\mathcal{E}^{\nu'}.
\end{equation}
Notice that  each row of $G^{\nu'}$ belongs to the (limiting) subdifferential of $u(x_i^{\nu'+1/2})$. %For here and in the following, we denote this relationship as $G^{\nu'}\in\partial U(X^{\nu'+1/2})$.

Further, by Lemma~\ref{lemma:block kl} (See Appendix A), we have that
\begin{equation}
    \label{eq:KL 3}
    \kappa|U(X^{\nu'+1/2})-\mathcal{L}^{\infty}|^0<\|G^{\nu'}\|,\quad \forall\nu'\in V'.
\end{equation}

Combining~\eqref{eq:G-D'} and~\eqref{eq:KL 3} yields
\begin{equation}
    \label{eq:constant bound}
    \mathcal{T}^{\nu'}\geq \frac{\kappa}{\tilde{c}_5}>0,\quad\forall\nu'\in V', U(X^{\nu'+1/2})\neq\mathcal{L}^{\infty}.
\end{equation} Using~\eqref{eq:lyapunov gap dyn},~\eqref{eq:constant bound}, and the fact that $\Delta\mathcal{L}^{\nu}\to 0$   monotonically, one infers that there exists $\nu_2'\in V'$ such that
\begin{equation}
    \label{eq:contract 3.2}
    \mathcal{T}^{\nu}=0,\quad\forall\nu\geq \nu_2'.
\end{equation}
\end{itemize}

We proceed combining the bounds in the three cases above. Set $\nu_2'=0$ if Case 3(i) happens and set $\nu_2=0$ if Case 3(ii) holds true. Let   $\nu_4>\nu_3>\max\{\nu_2,\nu_2'\}$ such that $\{\nu_3,\nu_3+1,\ldots, \nu_4-1\}\subset V$.  
Summing over such a set while using ~\eqref{eq:contract 1},~\eqref{eq:contract 2},~\eqref{eq:contract 3.1} and~\eqref{eq:contract 3.2}, we obtain: for any $\theta\in[0,1)$,
%\begin{equation}
 %   \label{eq:summability_case1}
%\sum_{\nu=\nu_3}^{\nu_4-1}\mathcal{T}^{\nu}\leq \frac{c_5\tau^{\nu_3}}{1-\tau},\quad \theta\in(0,1/2].
%\end{equation}
%\begin{equation}\label{eq:summability_case2}\begin{aligned}
 %   \sum_{\nu=\nu_3}^{\nu_4-1} \mathcal{T}^{\nu}\leq& \mathcal{T}^{\nu_3}\sum_{\nu=0}^{\nu_4-\nu_3-1}\left(\frac{1}{2}\right)^{\nu}+c_6\sum_{\nu=\nu_3}^{\nu_4-1}\sum_{j=0}^{\nu-\nu_3}\left(\frac{1}{2}\right)^j\Delta\mathcal{L}^{\nu-j}_{\theta}\\
%\leq&
%2\left[\mathcal{T}^{\nu_3}+c_6(\Delta\mathcal{L}^{\nu_3})^{1-\theta}\right],\quad\theta\in(1/2,1).
%\end{aligned}\end{equation}
%\begin{equation}
 %   \label{eq:summability_case3}
  %  \sum_{\nu=\nu_3}^{\nu_4-1} \mathcal{T}^{\nu}\leq0,\quad\theta=0.
%\end{equation}
\begin{equation}
    \label{eq:summability}
    \begin{aligned}
    \sum_{\nu=\nu_3}^{\nu_4-1}\mathcal{T}^{\nu}
    \leq&\max\left\{\frac{c_5\tau^{\nu_3}}{1-\tau},2\mathcal{T}^{\nu_3}+2c_6(\Delta\mathcal{L}^{\nu_3})^{1-\theta}\right\}.
    \end{aligned}
\end{equation}
We are ready to show that, once the  sequence $\{X^\nu\}$ enters a sufficiently small  neighborhood of $X^{\infty}$, it cannot    escape from it. Let    $r>0$ be small enough such that  
$$\mathcal{B}_{r}(X^{\infty}):=\{X\in\mathbb{R}^{m\times n}:\|X-X^{\infty}\|< r\}\subseteq\mathcal{N}_{\infty}.$$
By
Theorem~\ref{thm:1} and the fact that $X^{\infty}$ is an accumulation point of $\{X^\nu\}$,  
there exists   $V \ni\nu_3>\max\{\nu_2,\nu_2'\}$  such that  
$$\|D^{\nu_3}\|< \frac{r}{4},\quad \|X^{\nu_3+1/2}-X^{\infty}\|<\frac{r}{4}, \quad \text{and}$$
$$\max\left\{\frac{c_5\tau^{\nu_3}}{1-\tau},2\mathcal{T}^{\nu_3}+2c_6(\Delta\mathcal{L}^{\nu_3})^{1-\theta}\right\}<\frac{r}{2}.$$

Hence,  $$\|X^{\nu_3}-X^{\infty}\|<\frac{r}{2}, \quad  \text{and thus}\quad   X^{\nu_3}\in\mathcal{B}_r(X^{\infty}).$$  By~\eqref{eq:summability}, %hat: for any $\nu_4>\nu_3$ where $\nu_3$ up to $\nu_4-1$ belong to $V$, 
$$\sum_{\nu=\nu_3}^{\nu_4-1}\mathcal{T}^{\nu}<\frac{r}{2}.$$

We prove next by contradiction that,  for any $\nu\geq\nu_3$, $X^{\nu}\in\mathcal{B}_r(X^{\infty})\subseteq\mathcal{N}_{\infty}$. Let us assume the contrary:    there exists $\nu_4'>\nu_3$ (being the smallest index) such that $\|X^{\nu_4'}-X^{\infty}\|\geq r$. Then $\nu\in V$ for $\nu_3\leq\nu<\nu_4'$, and   
$$\begin{aligned}
    \|X^{\nu_4'}-X^{\infty}\|\leq&\|X^{\nu_4'}-X^{\nu_3}\|+\|X^{\nu_3}-X^{\infty}\|
    \leq \|X^{\nu_3}-X^{\infty}\|+\sum_{\nu=\nu_3}^{\nu_4'-1}\mathcal{T}^{\nu}<r.
    %<&\frac{r}{2}+\frac{r}{2}=r,
\end{aligned}$$
This contradicts the  assumption  $\|X^{\nu_4'}-X^{\infty}\|\geq r$. Since $r>0$ can be arbitrarily small, we have proved that $\{X^{\nu}\}$, as well as $\{X^{\nu+1/2}\}$, converge to $X^{\infty}.$%$=1(x^\star)^\top$.

\subsection{Convergence rate analysis}
We have shown both $\{X^{\nu}\}$ and $\{X^{\nu+1/2}\}$ converge to $X^{\infty}$, for any  accumulation point   $X^{\infty}$   of $\{X^{\nu}\}$. We  can now establish the convergence rate.

By the convergence of $\{X^{\nu+1/2}\}$, it follows that there exists $\nu_5$ such that $\forall\nu\geq\nu_5$, both~\eqref{eq:contract 1} and~\eqref{eq:contract 2} hold and either~\eqref{eq:contract 3.1} or~\eqref{eq:contract 3.2} holds. Notice that for any index  $\bar{\nu}\geq\nu_5$, it holds 
\begin{equation}
    \label{eq:potential cover 2}     \|X^{\nu}-X^\infty\|\leq
 \sum_{\nu=\bar{\nu}}^{\infty}\|X^{\nu+1}-X^{\nu}\|\leq \mathcal{S}^{\bar{\nu}}:=c_3\sum_{\nu=\bar{\nu} }^{\infty}\mathcal{T}^{\nu}. 
\end{equation}

 $\bullet$ {\bf Case 1: $\theta\in \left({1}/{2},1\right)$.} Substituting~\eqref{eq:Enu dynm} in~\eqref{eq:potential cover 2}, we have: 
 \begin{equation}\begin{aligned}
    \mathcal{S}^{\nu}\leq& \sum_{t=\nu}^{\infty}\frac{1}{2}\mathcal{T}^{t-1}+c_6\sum_{t=\nu}^{\infty}\Delta\mathcal{L}^{t}_{\theta}
    \leq  \frac{1}{2}\mathcal{S}^{\nu-1}+c_6(\Delta\mathcal{L}^{\nu})^{1-\theta}\\
   \overset{\eqref{eq:lyapunov Enu2}}{\leq}  &\frac{1}{2}\mathcal{S}^{\nu-1}+c_6\left(\tilde{c}_6^{\theta}\mathcal{T}^{\nu-1}\right)^{\frac{1-\theta}{\theta}}
\leq\mathcal{S}^{\nu-1}-\mathcal{S}^{\nu}+c_6(\tilde{c}_6^{\theta}/c_5)^{\frac{1-\theta}{\theta}}(\mathcal{S}^{\nu-1}-\mathcal{S}^{\nu})^{\frac{1-\theta}{\theta}},
\end{aligned}
\end{equation}
for all $\nu\geq\nu_5$.  Since $\lim_{\nu\rightarrow\infty} \mathcal{S}^{\nu}=0$ and $\frac{1-\theta}{\theta}<1$, there exists $\nu_6\geq\nu_5$ such that for all $\nu\geq\nu_6$, $\mathcal{S}^{\nu-1}-\mathcal{S}^{\nu}<1$. Hence,   
\begin{equation}
    \label{eq:sublinear}
(\mathcal{S}^{\nu})^{\frac{\theta}{1-\theta}}\leq c_9(\mathcal{S}^{\nu-1}-\mathcal{S}^{\nu}),\quad \forall\nu\geq\nu_6,
\end{equation}
where $c_9=1+c_6(\tilde{c}_6^{\theta}/c_5)^{\frac{1-\theta}{\theta}}.$ By Lemma~\ref{lemma:sublinear} (see Appendix~\ref{subsec:A}) and \eqref{eq:sublinear},  there exists $c_{10}>0$ such that 
$$\|X^{\nu}-1(x^*)^\top\|\leq \mathcal{S}^{\nu}\leq c_{10}\nu^{-\frac{1-\theta}{2\theta-1}},\quad \forall\nu\geq\nu_6.$$
Let $c_{11}:=\max_{\nu<\nu_6}\|X^{\nu}-1(x^*)^\top\|\cdot \nu_6^{\frac{1-\theta}{2\theta-1}}$, and pick $c:=\max\{c_{10},c_{11}\}$. Then,
$$\|X^{\nu}-1(x^*)^{\top}\|\leq c\nu^{-\frac{1-\theta}{2\theta-1}},\quad \forall\nu\geq 0.$$

$\bullet$ {\bf Case 2: $\theta\in (0, {1}/{2}]$}. By~\eqref{eq:contract 1},
\begin{equation*}
   % \label{eq:linear}
\|X^{\nu}-1(x^*)^\top\|\leq\mathcal{S}^{\nu}\leq c_7\tau^{\nu}, \quad\forall\nu\geq\nu_5,
\end{equation*}
for some $c_7>0$.
Let $c_8:=\max_{\nu\leq\nu_5}\|X^{\nu}-1(x^*)^\top\|\cdot\tau^{-(\nu_2+1)}$ and pick $c':=\max\{c_7,c_8\}$, then
$$\|X^{\nu}-1(x^*)^\top\|\leq c'\tau^{\nu},\quad\forall\nu\geq 0.$$

\indent$\bullet$ $\textbf{Case 3: }\theta=0$. We consider the following  two  sub-cases.
\begin{itemize}
    \item [\bf (i)]  \eqref{eq:contract 3.1} holds: Following the analysis as in Case 2,   we infer  that, there exists $c''>0$ such that
$$\|X^{\nu}-1(x^*)^\top\|\leq c''(\tau')^{\nu},\quad\forall\nu\geq 0$$
\item[\bf (ii)]  \eqref{eq:contract 3.2} holds:  $\{X^{\nu}\}$ will converges to $1(x^*)^\top$ in a finite number $\nu_2'$ of steps.
\end{itemize}

 \subsection{Final convergence results}
  The following  theorem   summarizes the  results proved in the previous sections. 
\begin{theorem}
    \label{thm:2}
     Consider Problem (\ref{problem}) under Assumption~\ref{ass:function}.   Let  $\{X^{\nu}\}_{\nu\in \mathbb{N}}$ be the sequence generated by  the SONATA algorithm  under Assumption~\ref{ass:W} and the   tuning of $\alpha$ and  $\gamma$ as in Theorem~\ref{thm:1}, with   $\xi=L$ therein. % \textcolor{blue}{Specifically, set $\xi=L$, $\alpha=\mathcal{O}(1/L)$ and $\gamma=\mathcal{O}(1/L(1-5\rho^2))$}. 
     Further  suppose  $\rho< {1}/{\sqrt{5}}$. 
  
%$$\omega:=\min\left\{\frac{\tilde{c}_2}{\tilde{c}_7},\frac{\tilde{c}_3}{\tilde{c}_8}\right\},$$
%where $\tilde{c}_2$, $\tilde{c}_3$ are defined in~\eqref{tc2},~\eqref{tc3} respectively and
%$$\begin{aligned}
%\tilde{c}_7&:=\kappa^{-1/\theta}(3L^2+3/\alpha^2)+14L_{\text{mx}}^2\rho^2\gamma,\\
%\tilde{c}_8&:=5\gamma\rho^{2}+3\kappa^{-\frac{1}{\theta}}+\max\left\{\frac{L}{L_{\text{mx}}^2}w_{\text{mx}},2L\alpha^2w_{\text{mx}}\right\}
%\end{aligned}$$
  Then, if the sequence $\{X^{\nu}\}$ has an accumulation point $X^{\infty}$--it mush be $X^{\infty}=1(x^*)^\top$, for some critical point $x^*$ of $u$--and $u$ satisfies the KL property at $x^\star$ with exponent $\theta\in (0,1)$ and coefficient $\kappa$, then $X^{\infty}$  is unique and
$$\lim_{\nu\rightarrow\infty}X^{\nu}=X^{\infty}=1(x^*)^\top.$$
 Further, the following convergence rates hold:
   \begin{enumerate}
       \item [(i)] If $\theta\in\left({1}/{2},1\right)$,   for all $\nu\in\mathbb{N}_+$ and some $c>0$,
       $$\|X^{\nu}-1(x^*)^{\top}\|\leq c\nu^{-\frac{1-\theta}{2\theta-1}};$$ 
       \item [(ii)]If $\theta\in(0,{1}/{2}]$, then for all $\nu\in\mathbb{N}_+$ and some $c^\prime>0$,
       $$\|X^{\nu}-1(x^*)^{\top}\|\leq c^\prime \left(\frac{1}{1+\omega}\right)^{\nu/2},$$ where {$$\omega=\mathcal{O}\left(\frac{L}{\kappa^{1/\theta}}+\frac{\rho^2}{1-5\rho^2}\frac{L_{\text{mx}}^2}{L^2}\right);$$}
       \item [(iii)] If $\theta=0$, then either $\{X^{\nu}\}$ converges to $1(x^*)^\top$ in a finite number of steps or there exists $c^{\prime\prime}>0$ such that for all $\nu\in\mathbb{N}_+$,
 $$\|X^{\nu}-1(x^*)^{\top}\|\leq c^{\prime\prime} \left(\frac{1}{1+\omega^\prime}\right)^{\nu/2},$$ where
{$$\omega'=\mathcal{O}\left(\frac{\rho^2}{1-5\rho^2}\frac{L_{\text{mx}}^2}{L^2}\right).$$}
   \end{enumerate}
   { The explicit expression of $\omega$ and $\omega^\prime$ above is given in~\eqref{omega} and~\eqref{omega'}, respectively. }
\end{theorem}

 Next we  provide the iteration complexity of the SONATA algorithm~\eqref{eq:a} under the KL property with   $\theta\in(0,1/2]$. %when $\rho$ is sufficiently small 
%\begin{corollary}
    %\label{cor:1}
    %If $\rho<\sqrt{\frac{3-\sqrt{5}}{20}}$. Choose $\alpha>0$ such that \begin{equation}
%\label{eq:alpha2}\begin{aligned}\alpha&<\min\left\{\frac{1}{2L},\frac{1}{11/8L+\mathcal{R}(L,L_{\text{mx}},\kappa^{1/\theta},\rho)}\right\},%\\ \rho^{2}&<\frac{3-\sqrt{5}}{20},
%\end{aligned}\end{equation}
%where
%$$\begin{aligned}
%&\mathcal{R}(L,L_{\text{mx}},\kappa^{1/\theta},\rho)\\%:=&30L_{mx}^2\rho^{2}\\
%&\cdot\left[(1+\frac{\kappa^{1/\theta}}{30L})K+\frac{10\rho^{2}}{1-10\rho^{2}}\cdot\frac{\kappa^{1/\theta}}{30L}K+\frac{\frac{1}{10L}+\frac{\kappa^{1/\theta}}{30L}Z}{1-10\rho^{2}}\right]\\
    %=&\mathcal{O}\left(\max\left\{\frac{L_{\text{mx}}^2}{L},\frac{L_{\text{mx}}^2}{L}\cdot\frac{\kappa^{2/\theta}}{L^2}\right\}\cdot\frac{\rho^{2}}{(1-10\rho^{2})^2}\right),
%\end{aligned}$$
%$$    Z:=\max\{\frac{L}{2L_{mx}^2}w_{mx},\frac{1}{2L}w_{mx}\}.$$
%Then there exist $k,w_1,w_2,\xi>0$ such that the inequalities \eqref{eq:C1},\eqref{eq:C2} and \eqref{eq:C3} hold and therefore both results in Theorem~\ref{thm:1} and Theorem~\ref{thm:2} hold true.\textcolor{red}{what's the point of this intermediate corollary in the main text if you do not provide the expression of the rate. It is incomplete. }\textcolor{blue}{Revised.}
%\end{corollary}
\begin{corollary}
    \label{cor:1} In state the setting of Theorem~\ref{thm:2}, with in particular   $\theta\in (0,1/2]$, and additionally
    %SONATA is run to generate sequence $\{X^{\nu}\}_{\nu\in \mathbb{N}}$, over a network under Assumption~\ref{ass:W} and $\rho<\frac{1}{\sqrt{5}}$. Suppose that the level set $\left[u\leq \mathcal{L}^0-(m-1)\underline{u}\right]$ is bounded. If further, we have that 
    \begin{equation}\label{eq:cond_rho}\rho<\min\left\{\frac{1}{\sqrt{5}},\frac{1}{\sqrt{42w_{\text{mx}}+26}}\cdot\frac{L}{L_{\text{mx}}}\right\}.\end{equation}
   Then, %with the same choice of parameters as in Theorem~\ref{thm:1}, 
   the number of iterations  needed for $\{X^\nu\}$ to reach an   $\epsilon$-stationary solution of \eqref{problem}  is given by $$\mathcal{O}\left(\frac{L}{\kappa^{1/\theta}}\log(1/\epsilon)\right),$$
   where $\kappa\in(0,\infty)$ is the KL coefficient of $u$ at $x^*$.
\end{corollary} 
%\subsection{Discussion} 
%\label{subsec:discussion}
Theorem~\ref{thm:2} fully characterizes the convergence of the sequence  
 the sequence $\{X^\nu\}$ generated by the SONATA algorithm,  under the KL property of $u$. This aligns with results achieved by the proximal gradient algorithm in centralized settings~\cite{attouch2009convergence}, and addresses the gap in the decentralized optimization literature.  The following comments are in order. \smallskip

$\bullet$ {\it On the condition on $\rho$}: 
The stipulation on $\rho$, as requested by \eqref{eq:cond_rho}, signifies the need of a
well-connected network. This is a common requirement by decentralized algorithms \cite{Nedic_Olshevsky_Rabbat2018}. When the network graph is predetermined
(with given $W$), one can meet the condition \eqref{eq:cond_rho} (if not a-priori satisfied) by employing at each
agent's side multiple rounds of communications per gradient evaluation.  This is a common practice
(see, e.g.,  \cite{scutari2019distributed}) that in the above setting results in a communication overhead of the order $\tilde{\mathcal{O}}((1-\rho)^{-1})$, if the same gossip matrix $W$ is used in each communication. This number can be reduced to   $\tilde{\mathcal{O}}((1-\rho)^{-1/2}$, if the gossip matrices in the communication rounds are chosen according to  Chebyshev \cite{auzinger2011iterative,scaman2017optimal} or   Jacobi  polynomials \cite{Berthier2020}.     Here,     $\rho=\|W-11^\top/m\|$  does not necessarily satisfies  \eqref{eq:cond_rho}.   \smallskip

$\bullet$ {\it On the linear convergence}:   When the extra communications per gradient evaluation are $\tilde{\mathcal{O}}((1-\rho)^{-1})$, the total number of communication to reach an $\epsilon$-stationary solution of \eqref{problem}, for arbitrary $\rho<1$ and KL exponent $\theta\in(0,1/2]$,
reads $$\tilde{\mathcal{O}}\left(\frac{L}{\kappa^{1/\theta}}\frac{1}{1-{\rho}} \log\left(\frac{1}{\epsilon}\right)\right).$$ Here, $\tilde{\mathcal{O}}$ hides logarithm dependence on $L, L_{\text{mx}}, \kappa$ and $\theta$.

  It is worth remarking that the convergence rate  established above recovers that in the literature (in particular of the SONATA algorithm \cite{scutari2019distributed}) when $f$ in (\ref{problem}) is assumed to be  $\mu$-strongly convex. In fact, in this case $u$   satisfies the KL property with exponent $\theta= {1}/{2}$ and $\kappa=\sqrt{\mu}$. Therefore,   $ {L}/{\kappa^{1/\theta}}$ reduces to $L/\mu$, which is the condition number of $f$. \smallskip 

$\bullet$ \textit{On the existence of accumulation points}: The postulate in Theorem~\ref{thm:2} on the existence of an accumulation point of the sequence $\{X^\nu\}$ is standard and holds under serval alternative conditions. For instance, boundedness of iterates can be guaranteed in the case   $u$ has compact sublevel sets $[u\leq \mathcal{L}^0-m\underline{u}]$, where $\mathcal{L}^0$ is defined in~\eqref{eq:Lyapunov_function}. \smallskip  
%\begin{remark}[Characterization of $\kappa$]\textcolor{red}{this remark cuts the flow of the discussion, it can go at the end when we comment the final theorem and rate}
%\label{remark:2}
   % Note that in the stronger case that $u$ is $\mu$-strongly convex, $u$ is KL with exponent $\theta=\frac{1}{2}$, we have that $\kappa\leq \sqrt{2\mu}$ and our result can recover the convergence rate of SONATA with a $\mu-$strongly convex objective function. Further, when $\theta\in(\frac{1}{2},1)$, we have that $\kappa\leq d^{1/2-\theta}\sqrt{2\mu}$.
%\end{remark}

$\bullet$ \textit{On the finite convergence when $\theta=0$: }Unlike the proximal gradient algorithm in the centralized setting,    finite time convergence may not always hold when $\theta=0$. This is due to the perturbation introduced by consensus errors. However, as by-product of our analysis,  we have identified two sufficient conditions under which  finite time convergence is guaranteed, namely:
 \begin{itemize} \item[{\bf (i)}] There exists a subsequence of $\{U(X^{\nu+1/2})\}$ such that $U(X^{\nu+1/2})-\mathcal{L}^{\infty}\geq 0$ over this subsequence; or
 \item[\textbf{(ii)}] The objective function satisfies a slightly stronger version of the KL property   at $x^*$, with exponent $\theta$, we refer to (Def.~\ref{def-KL-2} in) Appendix for the details. \end{itemize}  The condition in  \textbf{(i)} requires that  the stationary point $x^*$ the sequence is approaching to is a local minimum of $u$. For certain landscapes of $u$ this is the case;  we refer to~\cite{lee2019first,jin2017escape,lee2016gradient} and references therein  
 for some examples arising from the machine learning literature. Condition \textbf{(ii)} is actually satisfied by a wide range of applications. For instance,  this is the case for all   the applications discussed in Sec~\ref{sub-examples}; quite interestingly, they  share the same  exponents as those characterizing the   original KL property, as defined in Definition~\ref{def-KL}), see Appendix B. %Further, in the case that $\{X^{\nu}\}$ converges linearly, when the network is good enough, the iteration complexity will be independent with the function parameters $L$ and $\kappa^{1/\theta}$.
 
\subsection{Proof of Lemma~\ref{lemma:U KL}}
\label{subsec:lemma proof}
Since $X^{\infty}=1(x^*)^{\top}$ is an accumulation point, we have 
$\mathcal{L}^{\infty}=U(X^{\infty})$. By the KL assumption on $u$ and  Lemma~\ref{lemma:block kl}, we have that $U$ satisfies the KL property at $X^{\infty}$ with exponent $\theta$ and parameters ($\kappa/\max\{d^{\theta-1/2},1\},1)$. Hence,  there exists a neighborhood $\mathcal{N}_{\infty}$ of $X^{\infty}=1(x^\star)^\top$ such that,
for any $X\in\mathcal{N}^{\infty}$ satisfying $U(X)-\mathcal{L}^{\infty}<1$,
\begin{equation}
\label{eq:std KL 1}
(U(X)-\mathcal{L}^{\infty})^{\theta}\leq \left(\frac{\max\{d^{\frac{1}{2}},1\}}{\kappa}\right)\|G\|,\quad\forall G\in\partial U(X). 
\end{equation}
Here,   $\theta\in[0,1)$ and and $\kappa\in(0,\infty)$ are     the KL exponent and KL parameter of $u$   at the stationary point $x^*$, respectively. 

By Theorem~\ref{thm:1},  $$\mathcal{T}^{\nu}\rightarrow 0,\quad   U(X^{\nu+1/2})\rightarrow\mathcal{L}^{\infty}.$$It follows that, there exists ${\nu}_1\in\mathbb{N}_+$  such that
   \begin{equation*}
    \tilde{c}_5\mathcal{T}^{\nu}<1,\quad U(X^{\nu+1/2})-\mathcal{L}^{\infty}<1,\quad \forall \nu\geq{\nu}_1.
    \end{equation*}
Define 
$$V:=\left\{\nu\geq\nu_1: X^{\nu+1/2}\in\mathcal{N}_{\infty}\right\}.$$
Notice that $V$ is nonempty, due to   the fact that $X^{\infty}$ is an accumulation point of $\{X^\nu\}$. By~\eqref{eq:std KL 1}, it follows  that, for any $\nu\in V$ and any $G^{\nu}\in\partial U(X^{\nu+1/2})$, we have $\tilde{c}_5\mathcal{T}^{\nu}<1$ and,
\begin{equation}\label{eq:K-G}U(X^{\nu+1/2})-\mathcal{L}^{\infty}%\leq &\kappa^{-1}\left(\sum_{i=1}^m\|g_i\|^{1/\theta}\right)^{\theta}\\
\leq \left(\frac{\max\left\{d^{\theta-\frac{1}{2}},1\right\}}{\kappa}\right)^{1/\theta}\|G^{\nu}\|^{1/\theta},
\end{equation}
 for $\theta\in(0,1)$.
We proceed   upper bounding $\|G^{\nu}\|$ in terms of    $\mathcal{T}^{\nu}$ using Lemma~\ref{lemma:G-D} (See Appendix A): for any $\nu\in V$, there exists $G_0^{\nu}$ such that
\begin{equation}
    \label{eq:G-D}
    \|G_0^{\nu}\|^2
        \leq3\left(L^2+\frac{1}{\alpha^2}\right)\|D^{\nu}\|^2+3\mathcal{E}^{\nu}.
\end{equation}
Substitute~\eqref{eq:G-D} into~\eqref{eq:K-G} and we get our desired result. \hfill $\square$

%\subsection{Analysis of the applications satisfying KL in Def~\ref{def-KL}}
%\label{subsec:kl property}

 %With the above intermediate results, it suffices to show that all the applications we gave in Sec.~\ref{sub-examples} also satisfies our stronger definition of KL (Def.~\ref{def-KL}) with the same exponent as in~\cite{li2018calculus} except the Neural network, which does not matter since we do not have explicitly the value of $\theta$. In the following, we will check them all in details in the following:

%\begin{remark}
 %   The above proposition can cover lot's of applications including Sparse Linear Regression with $\ell_1$ regularization and logistic regression we have mentioned.
%\end{remark} 

   % \begin{remark}
    %    The above proposition covers several (non-convex) applications like leat squares problem with SCAD or MCP regularization and also the PCA we mentioned in our paper.
   % \end{remark}
%As for the phase retrieval problem, it satisfies an even stronger assumption than Def~\ref{def-KL} by~\cite{zhou2017characterization} and we do not need show it again. Hence, we have show that our definition for KL Def~\ref{def-KL} is still applicable for all the applications we mentioned.
\section{Numerical Results}
\label{sec:numeric result}
\noindent In this section, we present some numerical results that support our theoretical findings. All experiments were conducted on simulated undirected graphs consisting of $m=20$ nodes, with edges generated according to the Erdős-Rényi model with a connection probability of 0.45. The weight matrix 
$W$ used in all the decentralized algorithms was determined according to the Metropolis-Hastings rule~\cite{auzinger2011iterative}.  All the  decentralized algorithms in the simulations were randomly initialized within the unit ball, and their tuning parameters,  as reported, were determined through a grid-search procedure to achieve the best practical  performance.
\subsection{Decentralized PCA}
\label{subsec:pca}
 Consider the decentralized PCA problem, as described in Sec.~\ref{sub-examples}(iv). The data generation model is the following: %Our first application is the distributed PCA problem with $f(x) = - \sum_{i=1}^{m}||{D}_{i}{x}||^{2}$; $r(x) = I_{\mathcal{K}}(x)$ with $ \mathcal{K} = \{x\in\mathbb{R}^{d}: ||x||\le 1\}$;  $m=20, d=50$. 
 Each agent $i$ locally owns a data matrix ${A}_{i}=(a_{ij}^\top)\in \mathbb{R}^{20\times 50}$, whose rows $a_{ij}^\top$, $j=1\ldots, 20$, are i.i.d. random vectors, drawn from the $N({0},{\Sigma})$. The covariance matrix $\Sigma$, with eigendecomposition   ${\Sigma} = {U}{\Lambda}{U}^{\top}$, is generated as follows: synthesize ${U}$ first by generating a square matrix whose entries follow the i.i.d. standard normal distribution, then perform the QR decomposition to obtain its orthonormal basis; and the eigenvalues diag($\Lambda$) are i.i.d. uniformly distributed in $[0,1]$. 

%An undirected graph of $m = 20$ nodes is generated through Erdos-Renyi model with activating probability of 0.45 for each edge. The weight matrix $\widehat{{W}}$ is generated using the Metropolis-Hastings rule. 

  We benchmark SONATA (with tuning parameter  $\alpha = 0.001$) against  PProx-PDA \cite{DaMi17} (tuning parameters: $\gamma=10^{-4}, \ \beta=21.2$) and the  decentralized  ADMM~\cite{MZM16} (tuning parameters: ($\rho = 3$). The parameters for all algorithms were optimized through a grid search over a suitable range of values to ensure the best practical performance. In Figure \ref{fig2}, we plot  $\text{log}_{10}  ||{X}^{\nu}-1({x}^{*})^{\top}||$ versus the number of  iterations $\nu$, for the different algorithms.  All algorithm were observed to converge to the same stationary solution. The figure demonstrates the linear convergence rate of the SONATA algorithm, with SONATA comparing favorable to PProx-PDA and decentralized ADMM. It is important to remark that, unlike for SONATA, there is no theoretical support for the linear convergence for PProx-PDA and decentralized ADMM. %It is important to note that theoretical results for both these schemes provide far more pessimistic sublinear convergence behaviour than what is observed here empirically.  
As shown in the graph, the geometric convergence results match the theoretical results we had in this work.

\begin{figure}[h]
\centering
\hspace*{-0.01in}
\includegraphics[scale=0.32]{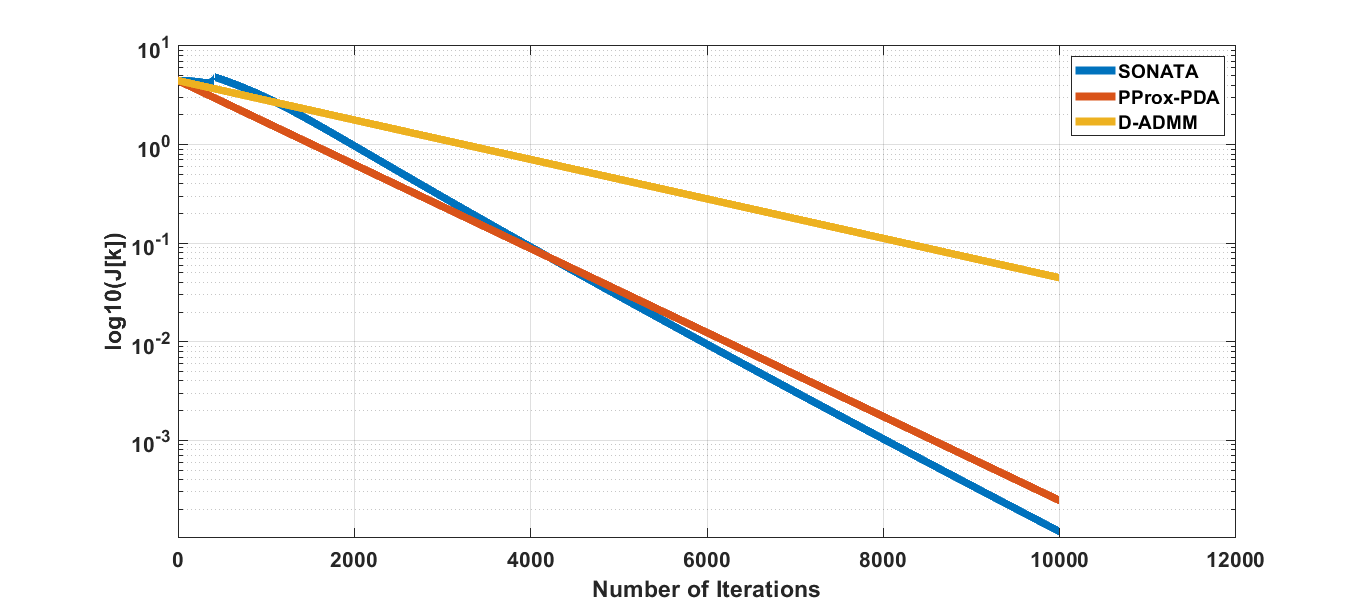}
\caption{Decentralized PCA: Distance of the iterates from a stationary solution vs. iterations.}
\label{fig2}
\end{figure}

\subsection{Sparse linear regression with $\ell_1$ regularization}
\label{subsec:lasso}
 Consider the decentralized LASSO problem using $\ell_1$ regularization, as described in Sec.~\ref{sub-examples}(i).  Data are generating according the following model.   The ground truth   ${x}^{*}\in \mathbb{R}^{350}$ is built   first drawing randomly a vector from the normal distribution $N({0},{I}_{350})$, then thresholding the smallest 40\% of its elements to $0$. The underlying linear model is ${y}_{i} = {A}_{i}{x}^{*} + {w}_{i}$, where the observation matrix ${A}_{i}\in \mathbb{R}^{15\times 350}$ is generated by   drawing i.i.d.  elements from the distribution $N(0,1)$, and then normalizing the rows to unit norm; and ${w}_{i}$ is the additive noise vector with i.i.d. entries from $N(0,0.1)$. Finally the regularization parameter is set to  $\lambda = 2$.
We contrast the  SONATA algorithm (tuning parameter: $\alpha = 0.015$) with the most popular decentralized algorithms proposed to solve convex, composite optimization problems, namely:  (i) NIDS (tuning parameters: $\alpha=0.017$)~\cite{ZWM19}, (ii) PG-EXTRA (tuning parameter: $\alpha=0.011$)~\cite{WQGW15}, (iii) Distributed ADMM (tuning parameter: $\rho = 4.2$)~\cite{MZM16}. We also included the   PProx-PDA (tuning parameters: $\gamma=10^{-4}, \ \beta=24.21$)~\cite{DaMi17}. 
%An undirected graph of 20 nodes is generated through Erdos-Renyi model with activating probability 0.45. The weight matrix $\widehat{{W}}$ is generated using the Metropolis-Hastings rule. 
 In Figure \ref{fig3}, we plot  $\text{log}_{10}   ||{X}^{\nu}-1({x}^{*})^{\top}||$ versus the number of   iterations $\nu$, where $x^\star$ is a solution of the LASSO problem.  The figure confirms linear convergence of the SONATA algorihtm, given that the LASSO function is a KL function with exponent $\theta=1/2$. No such a theoretical result is available for the other decentralized algorithms.  
\begin{figure}[h]
\hspace*{-0.01in}
\centering
\includegraphics[scale=0.32]{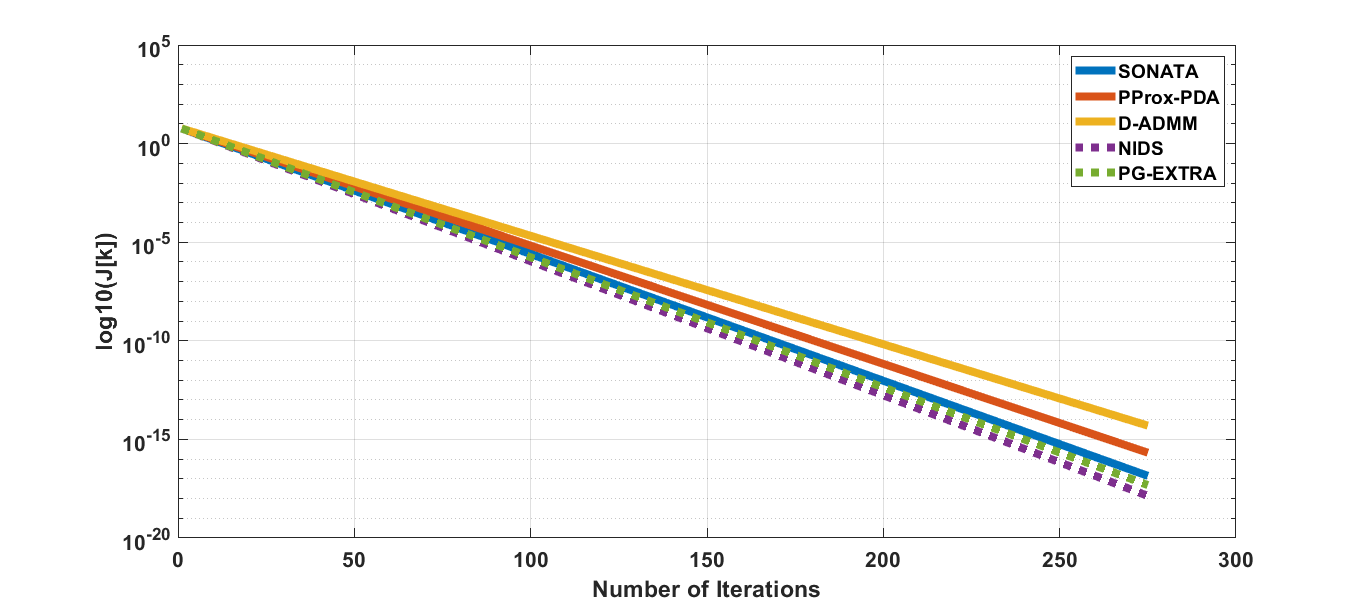}
\caption{LASSO with $\ell_1$ regularization: distance of the iterates from a solution vs. iterations. }
\label{fig3}
\end{figure}
\subsection{SCAD regularized distributed least squares}
\label{subsec:scad}
Consider now the sparse linear regression problem formulated as quadratic minimization with   the SCAD regularization, as described in  Sec.~\ref{sub-examples}(ii). The data generation model is the following. 
 The ground truth signal ${x}^{*}\in \mathbb{R}^{350}$ is built by first drawing randomly a vector from the normal distribution $\mathcal{N}({0},{I}_{350})$, then thresholding the smallest 20\% of its elements to $0$. The underlying linear model is generated as for the example in Sec.~\ref{subsec:lasso}.
%distributed least squares problem with $f({x}) = ||{A}{x}-{b}||^{2}$ , $m=20$, $d=350$, $a'=2.5,  \theta'=1$ (see Sec. \ref{subsec:KL} for details of $r$). The ground truth signal ${x}^{*}\in \mathbb{R}^{350}$ is built by first drawing randomly a vector from the normal distribution $\mathcal{N}({0},{I}_{n})$, then thresholding the smallest 20\% of its elements to 0. The underlying linear model is ${b}_{i} = {A}_{i}{x}^{*} + {n}_{i}$, where the observation matrix ${A}_{i}\in \mathbb{R}^{15\times 350}$ is generated by first drawing i.i.d. elements from the distribution $\mathcal{N}(0,1)$, and then normalizing the rows to unit norm and ${n}_{i}$ is additive noise with i.i.d. entries from $\mathcal{N}(0,0.1)$. 
 
%An undirected graph of 20 nodes is generated through Erdos-Renyi model with activating probability 0.45. The weight matrix $\widehat{{W}}$ is generated using the Metropolis-Hastings rule. 

The SONATA algorithm(with tuning: $\alpha = 0.002$) is compared with  the PProx-PDA (tuning parameter: $\gamma = 10^{-4}, \beta=30.23$)~\cite{DaMi17} and the  Distributed ADMM (tuning parameter $\rho = 2.5$)~\cite{MZM16}.
In Figure \ref{fig4}, we plot  $\text{log}_{10}   ||{X}^{\nu}-1({x}^{*})^{\top}||$ versus the number of  iterations $\nu$. All the algorithms are observed to converge to the same stationary solution.   The figure confirms the same results as observed for the other classes of functions satisfying the KL property with exponent $\theta=1/2$.
\begin{figure}[h]
\centering
\includegraphics[scale=0.32]{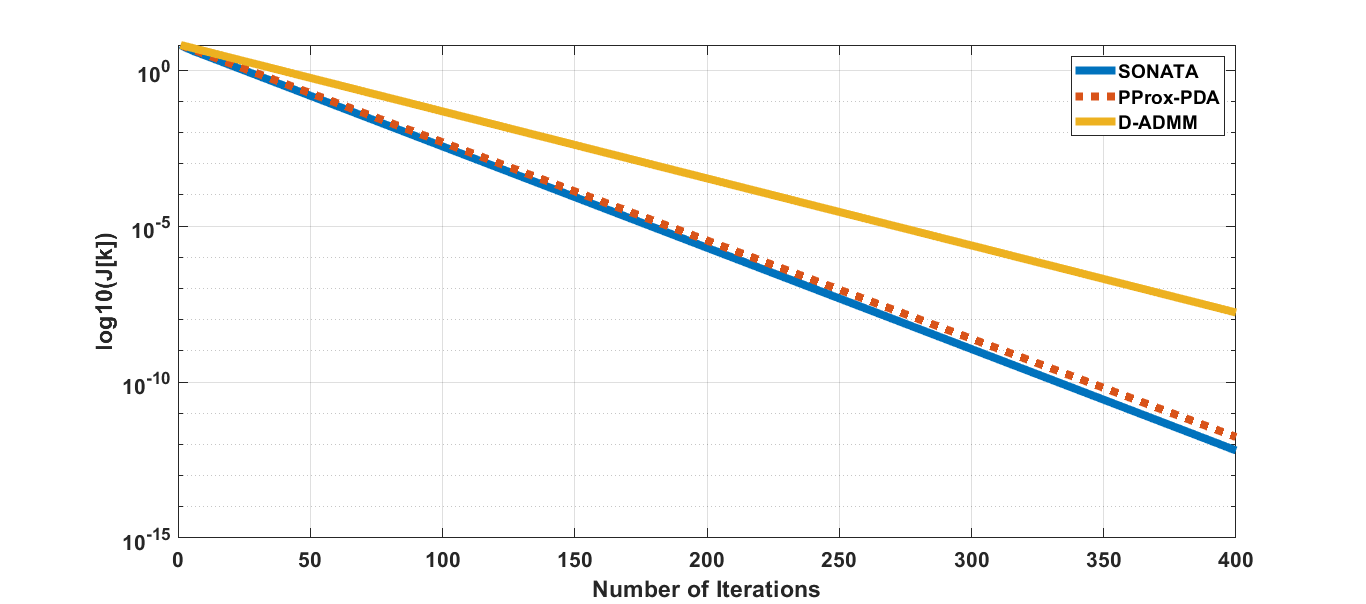}
\caption{Sparse linear regression with SCAD regularization: Distance from stationarity vs. iterations.}
\label{fig4}
\end{figure}
\vspace{-0.3cm}
\section{Concluding Remarks}
\vspace{-0.2cm}
We provided a new line of analysis for the rate of convergence of the SONATA algorithm, to solve Problem~\eqref{problem}, under the KL property of $u$.  The established convergence rates  match  those obtained in the centralized scenario by the centralized proximal gradient algorithm, for different values of the exponent $\theta\in (0,1)$. When $\theta=0$, finite time convergence of the sequence (as for the proximal gradient algorithm in the centralized case) cannot be always guarantees,  due to the perturbation induced by  consensus errors.  In such cases we still identified  sufficient conditions under which the SONATA algorithm matches the behavior of the centralized proximal gradient.  A future direction to  address this issue would be devising  decentralized designs that escape        local maxima.

\bibliographystyle{plain}
\bibliography{Bib}

\appendix

\section{Appendix for the rate analysis}
\label{subsec:A}
This appendix contains some definitions and lemmata established in the literature, and  used in our analysis of convergence rate. The proofs of the lemmata are presented for the sake of completeness.
 %  \begin{definition}[Subdifferential~\cite{rockafellar1998variational}]
%\label{def:subdifferential}
  %  For a proper function $f:\mathbb{R}^n\rightarrow(-\infty,\infty]$, the (limiting) subdifferential at $x\in\text{dom} f$ is defined by
   % $$\partial f(x):=\left\{v\in\mathbb{R}^n:\exists x^{t}\overset{f}{\rightarrow}x,v^{t}\rightarrow v, v^t\in\hat{\partial} f(x^t),\forall t \right\},$$
  %  where the $\hat{\partial} f(x)$ is the regular subdifferential of $f$ at $x\in\text{dom} f$,   given by
   % $$\hat{\partial} f(x):=\left\{v\in\mathbb{R}^n:\liminf_{z\rightarrow x,z\neq x}\frac{f(z)-f(x)-\langle v,z-x\rangle}{\|z-x\|}\geq 0\right\}.$$
%\end{definition} 
 
\begin{lemma}
    \label{lemma:ell_p equivalence}
   For $x\in\mathbb{R}^d$ and $\theta\in (0,1)$, the following hold: \begin{enumerate}
       \item[(i)] If $\theta\leq \frac{1}{2}$, then $\|x\|_{1/\theta}\leq \|x\|_2$;
       \item[(ii)] If $\theta>\frac{1}{2}$, then $\|x\|_{1/\theta}\leq d^{\theta-1/2}\|x\|_2.$
   \end{enumerate}
\end{lemma}
\begin{proof}
    By the H{\"o}lder's inequality, we have for any  $0<r<p$, \begin{equation}\label{eq:p_q_norm}\|x\|_p\leq\|x\|_r\leq d^{1/r-1/p}\|x\|_p.\end{equation}
   The inequalities in (i) and (ii)  follow from \eqref{eq:p_q_norm},  setting therein      $p= {1}/{\theta}$ and $r=2$ for (i),  and $p=2$, $r={1}/{\theta}$ for (i).  
\end{proof}
\begin{lemma}[Exponent for separable sums of KL functions]
    \label{lemma:block kl}
    Let $F(X)=\sum_{i=1}^mf_i(x_i)$, where $f_i$ is a proper closed function on $\mathbb{R}^d$ and $X=[x_1,x_2,\cdots,x_m]^\top$. Suppose further that for each $i$, $f_i$ is continuous on $\text{dom}\partial f_i$ and satisfies KL property (Def.~\ref{def-KL}) at $\bar x_i$ with exponent $\theta
    \in[0,1)$ and parameters $(\kappa_i,\eta_i)$. Then $F$ satisfies KL property (Def.~\ref{def-KL}) at $\bar X=[\bar x_1,\bar x_2,\cdots,\bar x_m]^\top$ with exponent $\theta$ and parameters $(\kappa/\max\{d^{\theta-\frac{1}{2}},1\},1)$ with $\kappa=\max_i\kappa_i$. 
\end{lemma}
\begin{proof}
See the proof of Theorem 3.3 in~\cite{li2018calculus} with minor variation and calculate the explicit expression of KL coefficient of $F$ using Lemma~\ref{lemma:ell_p equivalence}. 
\end{proof}
\begin{lemma}
    \label{lemma:MVT}
    For $0\leq b\leq a$, and $\theta<1$, we have that
    $$a-b\leq \frac{1}{1-\theta}a^{\theta}(a^{1-\theta}-b^{1-\theta}).$$
\end{lemma} 
\begin{proof}
    Define $\Phi(x):x^{1-\theta}$. Then, $\Phi'(x)=(1-\theta)x^{-\theta}$ is decreasing for $x>0$. By mean value theorem,  there exists $c\in[b,a]$ such that
    $$\begin{aligned}\Phi(a)-\Phi(b)= & \,\Phi'(c)(a-b)
    \geq \Phi'(a)(a-b)%=(1-\theta)a^{-\theta}(a-b)
    .\end{aligned}$$
    $$\Rightarrow a-b\leq \frac{1}{1-\theta}a^{\theta}(a^{1-\theta}-b^{1-\theta}).\vspace{-0.4cm} $$
\end{proof}
\begin{lemma}
    \label{lemma:G-D}
    Consider Problem~(\ref{problem}) under Assumption~\ref{ass:function}. Let $\{X^{\nu}\}_{\nu\in \mathbb{N}}$ be the sequence generated by the SONATA algorithm~\eqref{eq:A1}-\eqref{eq:A3}, under Assumption~\ref{ass:W}. Then there exists $G^{\nu} \in\partial U(X^{\nu+1/2}) $ such that 
\begin{equation*}
%\label{eq:G-D}
    \begin{aligned}
\|G^{\nu}\|^2
        \leq&3\left(L^2+\frac{1}{\alpha^2}\right)\|D^{\nu}\|^2+3\mathcal{E}^{\nu}.
    \end{aligned}
\end{equation*}
\end{lemma}
\begin{proof}
    By (\ref{eq:A1}), there exists $J^{\nu }\in\partial R(X^{\nu+1/2})$  such that 
 $$J^{\nu}+\frac{1}{\alpha}(X^{\nu+1/2}-X^{\nu}+\alpha Y^{\nu})=0.%\Rightarrow V^{\nu}=-\frac{1}{\alpha}(X^{\nu+1/2}-X^{\nu})-Y^{\nu},
$$  

% \textcolor{blue}{We need to define that it means $i$-th row vector of $V^{\nu}$ is a subgradient of $r(x_i^{\nu+1/2})$ }\textcolor{cyan}{Done.}
That is, there exists $G^{\nu} \in\partial U(X^{\nu+1/2}) $ such that 
\begin{equation*}
%\label{eq:G-D}
    \begin{aligned}
\|G^{\nu}\|^2=&\sum_{i=1}^m\left\|\nabla f(x_i^{\nu+1/2})-y_i^{\nu}-\frac{1}{\alpha}(x_i^{\nu+1/2}-x_i^{\nu})\right\|^2\\
        %=& \sum_{i=1}^m\|\nabla f(x_i^{\nu+1/2})-\nabla f(x_i^{\nu})+\delta_i^{\nu}-\frac{1}{\alpha}d_i^{\nu}\|^2\\
        %\leq& 3\sum_{i=1}^m(L^2\|d_i^{\nu}\|^2+\frac{1}{\alpha^2}\|d_i^{\nu}\|^2+\|\delta_i^{\nu}\|^2)\\
        \leq&3\left(L^2+\frac{1}{\alpha^2}\right)\|D^{\nu}\|^2+3\mathcal{E}^{\nu}.
    \end{aligned}
\end{equation*}
\end{proof}
\begin{lemma}[Th. 2 in \cite{attouch2009convergence}]
    \label{lemma:sublinear}
    Let  $\{\mathcal{S}^{\nu}\}$ be a nonnegative sequence satisfying    
$$(\mathcal{S}^{\nu})^{\frac{\theta}{1-\theta}}\leq c(\mathcal{S}^{\nu-1}-\mathcal{S}^{\nu}),\quad \forall \nu\geq \nu_0,$$  
 there exists $c'>0$ such that 
$$\mathcal{S}^{\nu}\leq c'\nu^{-\frac{1-\theta}{2\theta-1}}, \quad \forall \nu\geq\nu_0.$$
\end{lemma}
\section{Appendix for the KL property}
\label{subsec:B}
This appendix  introduces a stronger version of KL property   than that in Definition~\ref{def-KL-2},  which allows SONATA to match the convergence rate results of the centralized proximal gradient, when the KL exponent of $u$ (according to Definition~\ref{def-KL-2}) is zero.  Specifically, we have the following.
\begin{definition}{(A stronger version of the KL property)\cite{qiu2024kl}}\label{def-KL-2}
A proper closed function $f$ satisfies the KL property at  $\bar x\in \text{dom} f$   with exponent $\theta\in[0,1)$ if there exists a neighborhood $E$ of $\bar x$ and  numbers $\kappa,\eta\in (0,\infty)$ such that 
\begin{equation}
\label{eq353'}  \|g_x\|\geq \kappa|f(x)-f(\bar x)|^{\theta},
\end{equation} for all $x\in E\cap [f(\bar x)-\eta<f<f(\bar x)+\eta]$   and   $g_x\in\partial f(x)$. We call the function $f$ a KL function with exponent $\theta$ if it satisfies the KL property at any point $\bar x\in \text{dom} \partial f$ with the same exponent $\theta$.
\end{definition}

The relationship with the original KL property in  Definition~\ref{def-KL} follows readily. 

\begin{remark}
    \label{remark:KLs}
    If a function $f$ is proper and closed. Then $f$ satisfies KL property (Def.~\ref{def-KL-2}) at $\bar x\in\text{dom}\partial f$ with exponent $\theta\in[0,1)$ if and only if both $f$ and $-f$ satisfy the original KL property (Def.~\ref{def-KL})  at $\bar x$ with some exponents $\theta_1,\theta_2\in[0,1)$ respectively, where $\theta=\max\{\theta_1,\theta_2\}$.
\end{remark}

Several functions satisfy such a KL property. 
\begin{proposition}
    \label{prop:algebr}Let $f$ be a proper and closed sub-analytic function with closed
domain, then $f$ is a KL function with exponent $\theta\in[0,1)$.
\end{proposition}
\begin{proof}
By  in~\cite[Th.~3.1]{bolte2007lojasiewicz}, both $f$ and $-f$ satisfy the original KL property (Def.~\ref{def-KL}) at each $x\in\text{dom}\partial f$ with a $\theta\in[0,1)$.
\end{proof}
Furthermore,     such a stronger version of the KL     inherits most of the desirable  properties of that in Definition~\ref{def-KL}, including   exponent-preserving rules   such as composition rule \cite[Thm 3.2]{li2018calculus}, minimum rule \cite[Thm 3.1]{li2018calculus} and rule of separable sums\cite[Thm 3.3]{li2018calculus}. To be concrete, 
 with the help of several intermediate properties (Lemma~\ref{prop:nonstationary kl},~\ref{prop:luo-tseng},~\ref{prop:min KL}) proved below, we provide next two classes of functions satisfying Definition~\ref{def-KL-2} with exponent $\theta=1/2$ (see Theorem~\ref{thm:convex} and Theorem~\ref{thm:nonconvex} below). This also shows that  
all the illustrative examples discussed  in Sec.~\ref{sub-examples} satisfy Definition~\ref{def-KL-2} with the {\it same}  exponents $\theta$ as in Definition~\ref{def-KL}  (see Remark~\ref{remark:examples}). 
%This appendix is organized as follows: we provide the relationship between the stronger and original definition of KL (Remark~\ref{remark:KLs}) and address that a wide range of applications also satisfy the stronger version of KL (Prop~\ref{prop:algebr}). Finally,

Unless otherwise specified, in the following when calling  the KL property (or a KL function), we refer to  to the stronger Definition~\ref{def-KL-2}. 
 
\begin{lemma}
\label{prop:nonstationary kl}
    Suppose $f$ is a proper closed function, $\bar x\in\text{dom}$ $\partial f$ and $\bar x$ is not a stationary point ($0\notin\partial f(\bar x)$). Then for any $\theta\in[0,1)$, $f$ satisfies the KL property at $\bar x$ with exponent $\theta$.
\end{lemma}
\begin{proof}
    See Lemma 2.1 in~\cite{li2018calculus} with a minor variation.
\end{proof}
\begin{definition}[Luo-Tseng Error Bound]~\cite[Def 2.1]{li2018calculus}
\label{def:luo-tseng}
     Consider  Problem~\eqref{problem}. Let $\mathcal{X}$ be the set of stationary points of $u$. Suppose $\mathcal{X}\neq\emptyset$. We say that the Luo-Tseng error bound holds if for any $\zeta>\inf u$, there exists $c_1,\epsilon_1>0$ such that whenever $\|\texttt{prox}_r(x-\nabla f(x))-x\|<\epsilon_1$ and $u(x)<\zeta$,
\begin{equation*}
    \text{dist}(x,\mathcal{X})\leq c_1\|\texttt{prox}_r(x-\nabla f(x))-x\|
\end{equation*}
\end{definition}
\begin{lemma}[Luo-Tseng error bound implies KL]
\label{prop:luo-tseng}
Consider the problem~\eqref{problem} under Assumption~\ref{ass:function}. Suppose that the set  $\mathcal{X}$ of   stationary points of Problem~\eqref{problem} is nonempty and for any $x^*\in\mathcal{X}$, there exists $\delta>0$ such that $u(x)=u(x^*)$ whenever $x\in\mathcal{X}$ and $\|x-x^*\|\leq\delta$. If the Luo-Tseng error bound holds for $u$, then $u$ is a KL function with exponent $\theta={1}/{2}$.
\end{lemma}
\begin{proof}
    It follows for the proof of~\cite[Th. 4.1]{li2018calculus} with  minor variations.
\end{proof}
\begin{lemma}[Exponent of the minimum of finitely many KL functions]
\label{prop:min KL}
   Let $f_i$, $1\leq i\leq r$, be proper closed function with $\text{dom} f_i=\text{dom}$ $\partial f_i$ for all $i$, and $f:=\min_{1\leq i\leq r} f_i$ be continuous on $\text{dom} \partial f$. Suppose further that each $f_i$ is a KL function with exponent $\theta_i\in[0,1)$ for $1\leq i\leq r$. Then $f$ is also a KL function with exponent $\theta=\max\{\theta_i:1\leq i\leq r\}$.
\end{lemma}
\begin{proof}
    The proof follows similar steps as that in~\cite[Cor.~3.1]{li2018calculus}.
\end{proof}
%\begin{lemma}[Exponent for separable sums of Unifrom KL functions]
   % \label{prop:block kl}
    %Let $F(X)=\sum_{i=1}^mf_i(x_i)$, where $f_i$ is a proper closed function on $\mathbb{R}^d$ and $X=[x_1,x_2,\cdots,x_m]^\top$. Suppose further that for each $i$, $f_i$ is continuous on $\text{dom}\partial f_i$ and satisfies KL property  at $\bar x_i$ with exponent $\theta
    %\in[0,1)$ and parameter $\kappa_i\in(0,\infty)$. Then $F$ satisfies KL property at $\bar X=[\bar x_1,\bar x_2,\cdots,\bar x_m]^\top$ with exponent $\theta$ and parameter $\kappa/\max\{d^{\theta-\frac{1}{2}},1\}$ with $\kappa=\max_i\kappa_i$. 
%\end{lemma}
%\begin{proof}
 %   See Theorem 3.3 in~\cite{li2018calculus} with a minor variation and calculate the explicit expression of KL coefficient of $F$ by Lemma~\ref{lemma:ell_p equivalence}. 
%\end{proof}

We are ready to state the two major results of this section.   
\begin{theorem}[Convex problems with convex piecewise linear-quadratic regularizers]
\label{thm:convex}
Suppose $u$ is a proper closed function taking the form
$$u(x)=\min_{1\leq r\leq r}\underbrace{l(Ax)+r_i(x)}_{u_i(x)},$$
where $A\in\mathbb{R}^{m\times n}$, $r_i$ are proper closed polyhedral functions for $i=1,\cdots,r$ and $l$ satisfies either one of the following two conditions:
\begin{itemize}
    \item $l$ is a proper closed convex function with an open domain, and is strongly convex on any compact convex subset of $\text{dom} \,l$ and is twice continuously differentiable on $\text{dom} \,l$;
    \item $l(y)=\max_{z\in D}\left\{\right\langle y,z\rangle-q(z)\}$ for all $y\in\mathbb{R}^m$, with $D$ being a polyhedron and $q$ being a strongly convex differentiable function with a Lipschitz continuous gradient.
\end{itemize}
Suppose in addition that $f$ is continuous on $\text{dom}$ $\partial u$, then $u$ is a KL function with exponent $1/2$.
\end{theorem}
\begin{proof}
    Following the analysis in~\cite[Cor. 5.1]{li2018calculus}, one can show that, for each $i$, $u_i$ satisfies all the assumptions in Proposition~\ref{prop:luo-tseng}, and $\text{dom} u_i=\text{dom}\,\partial u_i.$
    The proof follows then  by Proposition~\ref{prop:nonstationary kl} (for the case $\text{argmin}$ $u_i=\emptyset$), Lemma~\ref{prop:luo-tseng} and Lemma~\ref{prop:min KL}.
\end{proof}
\begin{theorem}[nonconvex minimum-of-quadratic regularizers]
\label{thm:nonconvex}
Consider the following class of functions
$$u(x)=\min_{1\leq i\leq r}\underbrace{\left\{\frac{1}{2}x^\top A_i x+b_i^\top x+\beta_i+r_i(x)\right\}}_{u_i(x)},$$
 where $r_i$ are proper closed polyhedral functions, $A_i\in\mathbb{R}^{n\times n}$ is symmetric, $b_i\in\mathbb{R}^n$  and $\beta_i\in\mathbb{R}^n$ for $i=1,\cdots,r$. Suppose, in addition, that $u$ is continuous on $\text{dom}\,\partial u$. Then $u$ is a KL function with exponent $1/2$.
\end{theorem}
\begin{proof}
    Following the analysis in~\cite[Cor.~5.2]{li2018calculus}, we have that, for each $i$,
    $u_i(x)$ satisfies all assumptions in Proposition~\ref{prop:luo-tseng} and $\text{dom} u_i=\text{dom}\partial u_i.$  Then, the proof follows  combining Proposition~\ref{prop:nonstationary kl} (for the case $\text{argmin}$ $u_i=\emptyset$),~\ref{prop:luo-tseng} with Lemma~\ref{prop:min KL}.
    \end{proof}
\begin{remark}
\label{remark:examples}
    By Theorem~\ref{thm:convex}, we have that objective functions in the example (i) and (iii) in Sec.~\ref{sub-examples} are KL functions with exponent $\theta=1/2$. By Theorem~\ref{thm:nonconvex}, we have that objective functions of the example (ii) and (iv) in Sec.~\ref{sub-examples} are KL functions with exponent $\theta=1/2$. Objective function of example (v) in Sec.~\ref{sub-examples} is a KL function with exponent $1/2$, as  it satisfies the even stronger   gradient dominance condition in~\cite[Def.1]{zhou2017characterization}. Finally, by Proposition~\ref{prop:algebr}, the objective function of example (vi) in Sec~\ref{sub-examples} is a KL function with $\theta
    \in[0,1)$. This proves our claim that all the examples listed in Sec~\ref{sub-examples} share the same exponent $\theta$ in  both definitions (Def.~\ref{def-KL} and Def.~\ref{def-KL-2}).
\end{remark}

\end{document}